\documentclass[10pt]{article}
\usepackage[scale=0.8]{geometry}
\usepackage[english]{babel}
\usepackage{amssymb,amsmath,amsthm,bm}
\usepackage{amsfonts}
\usepackage{latexsym}
\usepackage{graphicx}
\usepackage{tikz}
\usetikzlibrary{fit,backgrounds}
\usepackage{lineno}
\usepackage{enumerate}
\usepackage{mathtools}
\usepackage[
pdfauthor={Le Chen and Songling Shan},
pdftitle={A characterization of edge-chromatic 3-critical graphs},
pdfstartview=XYZ,
colorlinks=true,
linkcolor=blue,
urlcolor=blue,
citecolor=blue,
bookmarks=true,
linktocpage=true,
hyperindex=true
]{hyperref}

\theoremstyle{plain}
\newtheorem{thm}{Theorem}[section]
\newtheorem{prop}[thm]{Proposition}
\newtheorem{lem}[thm]{Lemma}

\newtheorem{conj}[thm]{Conjecture}

\theoremstyle{remark}

\numberwithin{equation}{section}
\newcommand{\CC}{\mathcal{C}}
\DeclareMathOperator{\df}{def}

\newcommand{\pbar}{\overline{\varphi}}

\begin{document}
%\linenumbers
\title{A characterization of edge-chromatic 3-critical graphs}

\author{Le Chen\footnote{Department of Mathematics and Statistics, Auburn University, Auburn, AL 36849, USA.
Email: {\tt le.chen@auburn.edu}.
Supported in part by NSF grants DMS-2246850 and DMS-2443823 and by
a Collaboration Grant for Mathematicians (\#959981) from the Simons Foundation.}
\qquad
Songling Shan\footnote{Department of Mathematics and Statistics, Auburn University, Auburn, AL 36849, USA.
Email:  {\tt szs0398@auburn.edu}.
Supported in part by NSF grant   DMS-2451895.
}
}

\date{October 30, 2026}
\maketitle

\begin{abstract}
  A graph with maximum degree $\Delta$ is \emph{$\Delta$-critical} if it is
  connected, has chromatic index $\Delta+1$, and the deletion of any edge
  reduces its chromatic index to $\Delta$; it is \emph{nontrivial} if it
  has no subgraph $H$ with $|E(H)|>\Delta\lfloor |V(H)|/2\rfloor$. We prove a characterization of
  all nontrivial $3$-critical graphs: every such graph is obtained either
  from smaller $3$-critical graphs by replacing a vertex with a triangle,
  by a Haj\'os-join, or by replacing a vertex with a suitable class~$1$
  graph, or from a snark by deleting at most one vertex and some
  independent edges. The third operation, which we call a Meredith-type
  extension, generalizes the Meredith extension by allowing other
  class~$1$ graphs in place of $K_{2,3}$, and we give a sufficient
  condition under which it produces new nontrivial $3$-critical graphs
  from smaller ones. To the best of our knowledge, this is the first
  characterization of this kind for edge-chromatic $3$-critical graphs,
  and the Meredith-type extension has not previously been considered in
  this generality. The characterization was found by analyzing the
  output of an exhaustive computer search for nontrivial $3$-critical
  graphs of orders $13$ to $24$, and we make the resulting data publicly
  available. Based on these data, we also propose two conjectures that
  are closely related to the case $\Delta=3$ of Seymour's exact
  conjecture and to Gr\"otzsch's conjecture on the $3$-edge-colorability
  of planar graphs. 

  \medskip

\end{abstract}

\section{Introduction}\label{sec:introduction}

Let $G$ be a simple graph and let $k\ge 1$ be an integer.
An \emph{edge $k$-coloring} of $G$ assigns to each edge of $G$ one
of $k$ colors so that adjacent edges receive distinct colors, and the
\emph{chromatic index} $\chi'(G)$ is the least $k$ for which $G$ has an
edge $k$-coloring. In the 1960s, Vizing~\cite{vizing1964} and,
independently, Gupta~\cite{Gupta1966} showed that
$\Delta(G)\le\chi'(G)\le\Delta(G)+1$ for every graph $G$. Following
Beineke and Wilson~\cite{beineke1973}, a graph $G$ is of
\emph{class~$1$} if $\chi'(G)=\Delta(G)$ and of \emph{class~$2$} if
$\chi'(G)=\Delta(G)+1$. Holyer~\cite{Holyer} showed that it is NP-complete
to decide whether a graph is of class~$1$. Some graphs, however, are of
class~$2$ for a simple counting reason. For an integer $k$, a subgraph
$H$ of $G$ is \emph{$k$-overfull} if $|E(H)|>k\lfloor |V(H)|/2\rfloor$.
If $G$ is $\Delta(G)$-overfull, we simply say $G$ is \emph{overfull}. In an
edge coloring of $G$, each color class restricted to $H$ is a matching
with at most $\lfloor |V(H)|/2\rfloor$ edges. Hence, if $G$ has a
$\Delta(G)$-overfull subgraph $H$, then
\[
\chi'(G)\ge \chi'(H)\ge \frac{|E(H)|}{\lfloor |V(H)|/2\rfloor}>\Delta(G),
\]
and $G$ is of class~$2$.

A graph $G$ with $\Delta(G)=k$ is called \emph{edge-chromatic
$k$-critical} if it is connected, $\chi'(G)=k+1$, and $\chi'(G-e)=k$ for
every edge $e\in E(G)$. Since we are concerned only with edge colorings,
we write ``$k$-critical'' for edge-chromatic $k$-critical. Every class~$2$
graph of maximum degree $k$ contains a $k$-critical subgraph, so the
$k$-critical graphs are the minimal obstructions to $k$-edge-colorability
among graphs of maximum degree at most $k$.
A $k$-critical graph is \emph{trivial} if it contains a $k$-overfull
subgraph, and \emph{nontrivial} otherwise. There are no $1$-critical
graphs, and the $2$-critical graphs are exactly the odd cycles, all of
which are trivial. For $k=3$ the trivial graphs are easy to recognize: a
$3$-critical graph is trivial precisely when it is itself overfull, that
is, when it has odd order and exactly one vertex of degree~$2$, and it is
nontrivial precisely when it has at least three vertices of degree~$2$
(Lemma~\ref{lem:basic-3-critical}). The nontrivial $3$-critical graphs are
much less understood; they have been studied mainly through exhaustive
computer searches at small orders and through particular constructions,
which we review below.

The main result of this paper is a structural characterization of all
nontrivial $3$-critical graphs, Theorem~\ref{thm:characterization}. It
shows that every nontrivial $3$-critical graph either is built from
smaller $3$-critical graphs by one of three operations, or is obtained
from a snark by deleting at most one vertex and some independent edges. To
the best of
our knowledge, this is the first characterization of this kind for
edge-chromatic $3$-critical graphs. A second contribution is the third
of the three operations, the Meredith-type extension, which replaces a
vertex by a class~$1$ graph. It contains the classical Meredith extension
as the special case in which the inserted graph is $K_{2,3}$, but the
operation in this generality does not seem to have appeared in the
literature. We show in Theorem~\ref{thm:compatible-extension} that, under
a compatibility condition between the two graphs involved, it always
produces a new nontrivial $3$-critical graph from a smaller one, and in
Proposition~\ref{prop:kotzig} we give infinitely many graphs that can be
inserted into every nontrivial $3$-critical graph in this way. Thus the
operation provides a source of nontrivial $3$-critical graphs beyond the
known operations.

We now define the three operations. Vertex-blowups and Haj\'os-joins are
standard tools in the study of critical graphs; see, for
example,~\cite{BokalBrinkmannGrunewald2005}. The Meredith-type extension
replaces a vertex by a graph $H$; in the special case
$H=K_{\Delta-1,\Delta}$ it is known as a Meredith extension~\cite{MR311503},
and it was studied for critical graphs by Gr\"unewald and
Steffen~\cite{GruenewaldSteffen}.

\textbf{Vertex-blowup.}
Let $H$ be a subcubic graph and let $x\in V(H)$.
The graph $G$ obtained from $H-x$ and a disjoint triangle
 by joining each vertex of $N_H(x)$ to a distinct vertex
of the triangle is called a \emph{vertex-blowup} of $H$ at $x$.

\textbf{Haj\'os-join.}
Let $G_1$ and $G_2$ be vertex-disjoint graphs, and let
$u_1v_1\in E(G_1)$ and $u_2v_2\in E(G_2)$.
The graph $G$ obtained from $G_1-u_1v_1$ and $G_2-u_2v_2$
by identifying $u_1$ and $u_2$, and then adding the edge
$v_1v_2$, is called a \emph{Haj\'os-join} of $G_1$ and $G_2$.

\textbf{Meredith-type extension.}
Let $G$ and $H$ be two disjoint graphs, each with maximum degree
$\Delta$, where $H$ is connected. Suppose that $H$ has exactly $\Delta$ vertices of degree
$\Delta-1$, with every other vertex of degree $\Delta$; we call the
degree-$(\Delta-1)$ vertices the \emph{connector vertices} of $H$. Let
$y\in V(G)$ be a vertex of degree $\Delta$, with neighborhood
$N_{G}(y)=\{y_1,\ldots,y_\Delta\}$. The \emph{Meredith-type extension of
$G$ through $H$ at $y$} is the graph $G'$ obtained from the disjoint union
of $G-y$ and $H$ by adding a perfect matching of $\Delta$ edges between
the connector vertices of $H$ and the vertices $y_1,\ldots,y_\Delta$.

A \emph{snark} is a cubic class~$2$ graph that has no triangle and is
cyclically $4$-edge-connected, where, for an integer $k\ge 0$, a graph $G$
is \emph{cyclically $k$-edge-connected} if for every edge set $F$ of size
less than $k$, at most one component of $G-F$ contains a cycle. Two edges
are \emph{independent} if they do not share an end. Our main result
is the following.

\begin{thm}\label{thm:characterization}
    Let $G$ be a nontrivial $3$-critical graph of order $n$. Then one of the
    following holds.
    \begin{enumerate}[(a)]
        \item $G$ is obtained from a nontrivial $3$-critical graph of order
        $n-2$ through a vertex-blowup.
        \item $G$ is obtained from two smaller $3$-critical graphs, at least one
        of which is nontrivial, through a Haj\'os-join.
        \item $G$ is obtained from a smaller nontrivial $3$-critical graph
        by a Meredith-type extension through a class~$1$ graph.
        \item $n$ is odd and $G$ is obtained from a snark of order $n+1$ by
        deleting exactly one vertex and some independent edges.
        \item $n$ is even and $G$ is obtained from a snark of order $n$ by
        deleting some independent edges.
    \end{enumerate}
\end{thm}

Much of the interest in small critical graphs comes from the Critical
Graph Conjecture, proposed by Beineke and Wilson~\cite{beineke1973} in 1973
and independently by Jakobsen~\cite{jakobsen1974a} in 1974, which asserts
that every edge-chromatic critical graph has odd order.
Goldberg~\cite{Goldberg1981} constructed an infinite family of
$3$-critical graphs of even order, the smallest of which has $22$
vertices, so the conjecture fails for $3$-critical graphs, and Chetwynd and
Fiol independently found $4$-critical graphs of even order; see the
survey~\cite{HiltonWilsonProgress}.

Small critical graphs have been determined order by order. In the paper
cited above, Jakobsen~\cite{jakobsen1974a} listed all $3$-critical graphs
on at most $10$ vertices: they have order $5$, $7$, or $9$, and the graph
obtained from the Petersen graph by deleting a vertex is the only
nontrivial one. Beineke and Fiorini~\cite{BeinekeFiorini1976} determined
the critical graphs on at most $7$ vertices and proved that there is no
$3$-critical graph of order $12$. Based on these results,
Fiorini and Wilson~\cite{FioriniWilson1977} listed the critical graphs of
order at most $8$ and noted that there is no critical graph of even
order at most $10$, and Chetwynd and Yap~\cite{ChetwyndYap1983}
determined the critical graphs of order $9$ with maximum degree at
least $4$. As a
result, the Petersen graph minus a vertex is the only nontrivial critical
graph on at most $10$ vertices. With computer assistance, Brinkmann and
Steffen showed that there are exactly two nontrivial critical graphs of
order $11$ and none of order $12$~\cite{BrinkmannSteffen1998}, and that the
smallest $3$-critical graphs of even order have $22$ vertices, with
exactly one of order $22$ and exactly nine of order
$24$~\cite{BrinkmannSteffen1997}; they also showed that the smallest
$4$-critical graphs of even order have $18$ vertices. Bokal, Brinkmann, and
Gr\"unewald~\cite{BokalBrinkmannGrunewald2005} showed that there are
exactly $14$ nontrivial critical graphs of order $13$, all of them
$3$-critical, and none of order $14$. These results were proved partly by
theoretical arguments and partly by computation.

To the best of our knowledge, the number of nontrivial $3$-critical graphs
of odd order greater than $13$ has not been determined. Using an
exhaustive computer search, we determine it for the orders $15$, $17$,
$19$, $21$, and $23$, and our computations also reproduce the known values
for the orders $13$ and $14$ and for the even orders from $16$ to $24$.
Table~\ref{table:enumeration} summarizes what is known. In
Appendix~\ref{S:algorithmic} we describe the search and classify the
graphs found according to the first case of the proof of
Theorem~\ref{thm:characterization} that applies; at every odd order from $13$ to $23$,
most of them fall under the vertex-blowup case (a).

\begin{table}[h]
\centering
\begin{tabular}{c|c|l}
\hline
Order $n$ & \# Nontrivial $3$-critical graphs & Reference \\ \hline
4,5,6,7,8  & 0 & \cite{jakobsen1974a} \\
9  & 1 & \cite{jakobsen1974a} \\
10 & 0 & \cite{jakobsen1974a,FioriniWilson1977} \\
11 & 2 & \cite{BrinkmannSteffen1998} \\
12 & 0 & \cite{BeinekeFiorini1976,BrinkmannSteffen1998} \\
13 & 14 & \cite{BokalBrinkmannGrunewald2005}; reproduced here \\
14 & 0 & \cite{BokalBrinkmannGrunewald2005}; reproduced here \\
15 & 94 & this paper \\
16 & 0 & \cite{BrinkmannSteffen1997}; reproduced here \\
17 & 774 & this paper \\
18 & 0 & \cite{BrinkmannSteffen1997}; reproduced here \\
19 & $6{,}984$ &  this paper \\
20 & 0 & \cite{BrinkmannSteffen1997}; reproduced here \\
21 & $70{,}530$ &  this paper \\
22 & 1 & \cite{BrinkmannSteffen1997}; reproduced here \\
23 & $782{,}186$ &  this paper \\
24 & 9 & \cite{BrinkmannSteffen1997}; reproduced here \\
\hline
\end{tabular}
\caption{Number of known nontrivial $3$-critical graphs by order}
\label{table:enumeration}
\end{table}

The census also suggests that cases~(d) and~(e) of
Theorem~\ref{thm:characterization} can be sharpened. For every odd order
up to $23$, each nontrivial $3$-critical graph that falls under case~(d),
in particular each one that is not obtained from smaller $3$-critical graphs
by any of the three operations, has exactly three vertices of degree~$2$,
and so it is obtained from a snark by deleting a single vertex and no
edge. Likewise, the unique nontrivial $3$-critical graph of order $22$,
which falls under case~(e), has exactly four vertices of degree~$2$, and
so it is obtained from a snark by deleting exactly two independent edges;
the same holds for the only triangle-free one among the nine nontrivial
$3$-critical graphs of order $24$. See Appendix~\ref{S:census-map} for
details. We propose the following conjecture, noting that for even
orders it rests on these two graphs only.

\begin{conj}\label{conj:snark-cases}
Let $G$ be a nontrivial $3$-critical graph that is not obtained from
smaller $3$-critical graphs by a vertex-blowup, a Haj\'os-join, or a
Meredith-type extension. If $|V(G)|$ is odd, then $G$ is obtained from a
snark by deleting exactly one vertex; if $|V(G)|$ is even, then $G$ is
obtained from a snark by deleting exactly two independent edges.
By Theorem~\ref{thm:characterization}, this is equivalent to $G$ having
exactly three vertices of degree $2$ if $|V(G)|$ is odd, and exactly four
if $|V(G)|$ is even.
\end{conj}

Let $P^*$ denote the graph obtained from the Petersen graph by deleting
one vertex; as noted above, it is the only nontrivial $3$-critical graph
of order at most $10$. Each of the three operations, applied to
$3$-critical graphs, produces a graph that contains the graph it is
applied to as a minor: contracting the new
triangle undoes a vertex-blowup, contracting the inserted graph $H$ undoes
a Meredith-type extension, and in a Haj\'os-join of $G_1$ and $G_2$,
contracting the part coming from $G_2$ (which is connected, since $G_2$
is $2$-connected by Lemma~\ref{lem:basic-3-critical}(1)) onto the
identified vertex yields $G_1$. Consequently,
every nontrivial $3$-critical graph built from $P^*$ by repeated
applications of the three operations contains $P^*$ as a minor. We
verified that every nontrivial $3$-critical graph of order at most $24$
contains $P^*$ as a minor (Appendix~\ref{S:census-map}), and we propose
the following.

\begin{conj}\label{conj:petersen-minor}
Every nontrivial $3$-critical graph contains $P^*$ as a minor.
\end{conj}

Since $P^*$ is not planar, Conjecture~\ref{conj:petersen-minor} implies
that every $3$-critical planar graph is trivial. It therefore implies
Gr\"otzsch's conjecture (attributed to Gr\"otzsch by
Seymour~\cite{Seymour1981}; see also~\cite[Section~12.16]{JensenToft1995}
and~\cite[Conjecture~15]{GoetzeJungeblutUeckerdt2022}) and Seymour's exact
conjecture~\cite{Seymour1979,Seymour1981} for simple planar graphs of
maximum degree~$3$.  Both conjectures imply  the  Four Color Theorem. 
Gr\"otzsch's conjecture states that every
$2$-connected planar graph with maximum degree at most $3$ is
$3$-edge-colorable unless it has exactly one vertex of degree~$2$, and
Seymour's exact conjecture asserts that every planar multigraph $G$
satisfies $\chi'(G)=\max\{\Delta(G),\lceil\Gamma(G)\rceil\}$, where
$\Gamma(G)=\max\, 2|E(G[X])|/(|X|-1)$ over all $X\subseteq V(G)$ of odd
size at least~$3$ (when $G$ is simple with maximum degree 3, $\Gamma(G)>3$ implies that $G$ contains a $3$-overfull subgraph). 

Tutte conjectured that every bridgeless cubic graph without a Petersen
minor is $3$-edge-colorable (see~\cite{RobertsonSeymourThomas1997}).
Robertson, Sanders, Seymour, and Thomas announced a proof of this
conjecture. Robertson, Seymour, and
Thomas~\cite{RobertsonSeymourThomas1997} reduced it to apex graphs and
doublecross graphs (see
also~\cite{RobertsonSeymourThomas2017,RobertsonSeymourThomas2019}), and
the doublecross case was settled
in~\cite{EdwardsSandersSeymourThomas2016}; the apex case, whose proof
was announced but not published, has recently been given by other
authors in the preprint~\cite{InoueKawarabayashiMatsuoMiyashitaMoharSonobe2026}.
Together with Tutte's conjecture, Conjecture~\ref{conj:snark-cases}, which
sharpens cases (d) and (e) of Theorem~\ref{thm:characterization}, would
imply that every smallest nontrivial $3$-critical graph that has no
$P^*$ minor has even order and is obtained from a snark by deleting
exactly two independent edges. Indeed, let $G$ be such
a graph, and suppose that $G$ is obtained from smaller $3$-critical
graphs by one of the three operations. At least one of these smaller
graphs is nontrivial: a vertex-blowup and a Meredith-type extension
applied to a $3$-critical graph do not change the number of vertices of
degree $2$, and at least one of the
two graphs of a Haj\'os-join is nontrivial (see the paragraph before
Theorem~\ref{lem:Hajos2}). By the discussion before
Conjecture~\ref{conj:petersen-minor}, $G$ contains this smaller
nontrivial $3$-critical graph as a minor, which therefore has no $P^*$
minor, contradicting the choice of $G$.
Hence Conjecture~\ref{conj:snark-cases} applies to $G$. If $|V(G)|$
were odd, then $G=H-v$ for a snark $H$ and a vertex $v\in V(H)$; the
snark $H$ is a bridgeless cubic graph that is not $3$-edge-colorable,
so it contains the Petersen graph as a minor, and deleting from a model
of this minor the branch set containing $v$, or any branch set if $v$
lies in none, gives a model of $P^*$ in $G$, so that $G$ would contain
$P^*$ as a minor, a contradiction. Hence
$|V(G)|$ is even, and $G$ is obtained from a snark by deleting exactly
two independent edges. In particular, under these assumptions,
if every nontrivial $3$-critical graph obtained from a snark by deleting
two independent edges contains $P^*$ as a minor, then
Conjecture~\ref{conj:petersen-minor}, and hence Gr\"otzsch's conjecture,
holds.

\paragraph{Organization.}
Section~\ref{S:operations} collects preliminaries on edge colorings and
small edge-cuts, establishes basic facts about $3$-critical graphs, and
studies the three operations. Vertex-blowups and Haj\'os-joins are treated
in Lemmas~\ref{lem:blow-up} and~\ref{lem:Hajos1} and
Theorem~\ref{lem:Hajos2}. For Meredith-type extensions,
Theorem~\ref{thm:compatible-extension} gives sufficient conditions for
$3$-criticality, illustrated by three examples,
Proposition~\ref{prop:kotzig} supplies an infinite family of graphs $H$
for which these conditions always hold, and Theorem~\ref{lem:cyclic-cut-reduction} gives
the reduction used for case (c) of Theorem~\ref{thm:characterization}. Section~\ref{S:characterization}
proves Theorem~\ref{thm:characterization}; the snark in cases (d) and (e)
is constructed in Lemma~\ref{lem:snark-reduction}, whose proof occupies
Subsection~\ref{SS:snark-proof}. Appendix~\ref{S:algorithmic} describes the
computer search behind Table~\ref{table:enumeration}.

\paragraph{Conventions and notation.}
All graphs are finite and simple unless explicitly stated otherwise. For
integers $p$ and $q$, let $[p,q]=\{i\in\mathbb{Z}:p\le i\le q\}$, and
write $[q]$ for $[1,q]$. Further notation is introduced in
Subsection~\ref{SS:prelim}.

\section{Preliminaries and properties of the three operations}\label{S:operations}

In this section we fix notation, record basic facts about $3$-critical
graphs and about small edge-cuts, and study how the three operations of
Section~\ref{sec:introduction} interact with $3$-criticality.

\subsection{Preliminaries}\label{SS:prelim}

Let $G$ be a graph. We use $V(G)$ and $E(G)$ to denote the vertex set
	and edge set of $G$, respectively, and write $e(G)=|E(G)|$.
	For $v\in V(G)$, let $N_G(v)$ be the set of neighbors of $v$ in $G$. The degree
	of $v$ in $G$, denoted by $d_G(v)$, is the number of edges of $G$ incident with
	$v$.

	Let $V_1,V_2\subseteq V(G)$ be disjoint vertex sets. We write $E_G(V_1,V_2)$ for the set of edges of $G$ with one endpoint in $V_1$ and the other in $V_2$, and let $e_G(V_1,V_2)=|E_G(V_1,V_2)|$. If $V_1=\{v\}$, we write $E_G(v,V_2)$ and $e_G(v,V_2)$ instead of $E_G(\{v\},V_2)$ and $e_G(\{v\},V_2)$.

Let $S\subseteq V(G)$ and $v\in V(G)$. We write $d_G(v,S)=e_G(v,S\setminus\{v\})$. The subgraph of $G$ induced by $S$ is denoted by $G[S]$. We write $G-S=G[V(G)\setminus S]$, and use the abbreviation $G-x$ for $G-\{x\}$. We also write $e_G(S)=e(G[S])$.

	If $F\subseteq E(G)$, then $G-F$ denotes the graph obtained from $G$ by deleting all edges in $F$.
	For a set $F \subseteq E(\overline{G})$ of edges, let $G+F$ denote the graph obtained from $G$ by adding all edges in $F$.
	We also use the abbreviation $G+e$ for $G+\{e\}$.

Let $k\in\mathbb{N}$, and let $\CC^k(G)$ denote the set of all edge $k$-colorings of a multigraph $G$. For $\varphi\in \CC^k(G)$ and $v\in V(G)$, the set of colors \emph{present} at $v$ is
	\[
	\varphi(v)=\{\varphi(e): \text{$e\in E(G)$ is incident with $v$}\},
	\]
	and the set of colors \emph{missing} at $v$ is $\pbar(v)=[k]\setminus \varphi(v)$. For $X\subseteq V(G)$ and $i\in [k]$, define $\pbar_X^{-1}(i)=\{v\in X:i\in \pbar(v)\}$. We write simply $\pbar^{-1}(i)$ for $\pbar_{V(G)}^{-1}(i)$.

An edge $e$ of a graph $G$ is \emph{critical} if $\chi'(G-e)<\chi'(G)$;
thus every edge of a $k$-critical graph is critical. We use the following
adjacency lemma of Vizing.

\begin{lem}[Vizing's Adjacency Lemma (VAL), \cite{Vizing-2-classes}]\label{lem:val}
	Let $G$ be a class~$2$ graph with maximum degree $\Delta$. If $e=xy$ is a critical edge of $G$, then $x$ has at least $\Delta-d_G(y)+1$ neighbors of degree $\Delta$ in $V(G)\setminus\{y\}$. As a consequence, $d_G(x)+d_G(y)\ge \Delta+2$.
\end{lem}

    Given an edge coloring of $G$ and a color $i$, the edges colored $i$ form a
matching, and hence the number of vertices not incident with an edge of color
$i$ has the same parity as $|V(G)|$. We will use the following Parity Lemma of
Gr\"unewald and Steffen~\cite[Lemma~2.1]{MR2028248}.

\begin{lem}[Parity Lemma]\label{lem:parity-lemma}
	Let $G$ be a multigraph, and let $\varphi \in \CC^k(G)$  for some integer
	$k\ge \Delta(G)$. Then $|\pbar^{-1}(i)|\equiv |V(G)| \pmod 2$ for every
	color $i\in [k]$.
\end{lem}

Let $G$ be a graph and let
 $F\subseteq E(G)$. We call $F$ a \emph{disconnecting set} if
$G-F$ has more components than $G$.  Furthermore, if there exists a set
$S\subseteq V(G)$ such that $F=E_G(S,V(G)\setminus S)$,
then $F$ is called an \emph{edge-cut}. For every $S\subseteq V(G)$ we
write $\delta_G(S)=E_G(S,V(G)\setminus S)$, and we omit the subscript $G$
when the graph is clear from the context. It is easy to verify that every inclusion-minimal
disconnecting set is an edge-cut.
A disconnecting set $F$ is \emph{essential} if at least two components of $G-F$ each contain an edge, and is \emph{cyclic} if at least two components of $G-F$ each contain a cycle.
Recall from Section~\ref{sec:introduction} that $G$ is
\emph{cyclically $k$-edge-connected} if for every edge set $F$ with
$|F|<k$, at most one component of $G-F$ contains a cycle; for a connected
graph $G$ this says precisely that $G$ has no cyclic disconnecting set
with fewer than $k$ edges.

We next record basic facts about $3$-critical graphs.

\begin{lem}\label{lem:basic-3-critical}
Let $G$ be a $3$-critical graph. Then the following hold.
\begin{enumerate}[(1)]
    \item $G$ is $2$-connected, and so every vertex of $G$ has degree at
    least $2$.
    \item $G$ is not regular and does not contain exactly two vertices of
    degree $2$.
    \item No two vertices of degree $2$ in $G$ are adjacent, and no two
    vertices of degree $2$ in $G$ have a common neighbor.
    \item $G$ is trivial if and only if $G$ has exactly one vertex of
    degree $2$. In this case $|V(G)|$ is odd and $G$ itself is overfull.
    \item $G$ is nontrivial if and only if $G$ has at least three vertices
    of degree $2$.
    \item If $S\subseteq V(G)$ satisfies $|\delta_G(S)|\le 3$ and $G[S]$
    is acyclic, then $S$ contains at most one vertex of degree $2$ in
    $G$. In particular, every essential disconnecting set of $G$ with two
    edges is a cyclic disconnecting set.
\end{enumerate}
\end{lem}

\begin{proof}
Part (1) follows from the $3$-criticality of $G$, part (2) is proved by
Jakobsen~\cite[Theorems~1.2 and~1.6]{jakobsen1974a} (note that Jakobsen
calls a graph $4$-critical when it is critical with chromatic index $4$,
so his $4$-critical graphs are our $3$-critical graphs), and part (3)
follows from Vizing's Adjacency Lemma (Lemma~\ref{lem:val}).

For (4) and (5), since $G$ is $3$-critical, every proper subgraph of
$G$ is $3$-edge-colorable and hence not $3$-overfull. Thus $G$ is trivial if
and only if $G$ itself is overfull. As every vertex of $G$ has degree
$2$ or $3$ by (1), $G$ is overfull if and only if $G$ has exactly one
vertex of degree $2$, in which case $|V(G)|$ is odd. This proves (4).
Since $G$ has at least one and not exactly two vertices of degree $2$ by
(2), part (5) follows from (4).

For (6), we may assume that $S\ne\emptyset$. Let $q$ be the number of
components of $G[S]$. Since $G[S]$ is acyclic, $e_G(S)=|S|-q$, and so
\[
\sum_{v\in S}\bigl(d_G(v)-2\bigr)
=2e_G(S)+|\delta_G(S)|-2|S|
=|\delta_G(S)|-2q.
\]
The left-hand side is nonnegative by (1), and $|\delta_G(S)|\le 3$, so
$q=1$ and the left-hand side is at most $1$. Hence $G[S]$ is a tree,
and all vertices of $S$ except at most one have degree $2$ in $G$.
Suppose, for a contradiction, that $u$ and $v$ are two vertices of $S$
that have degree $2$ in $G$. By (3), the path from $u$ to $v$ in the tree
$G[S]$ has length neither $1$ nor $2$. Its length is not $3$ or more
either, for then it would contain at least four vertices, all but at
most one of which have degree $2$ in $G$, and hence two adjacent vertices
of degree $2$ in $G$, again contradicting (3). This proves the first
assertion of (6).

For the second assertion, let $F$ be an essential disconnecting set of
$G$ with $|F|=2$, and let $D$ be a component of $G-F$ that contains an
edge. We show that $D$ contains a cycle; since $F$ is essential, this
shows that $F$ is cyclic. Every edge of $G$ between $V(D)$ and
$V(G)\setminus V(D)$ lies in $F$, so $\delta_G(V(D))\subseteq F$. Since
$G$ is $2$-edge-connected by (1) and $\emptyset\ne V(D)\ne V(G)$, we have
$|\delta_G(V(D))|\ge 2=|F|$, and hence $\delta_G(V(D))=F$. In particular,
no edge of $F$ has both ends in $V(D)$, so $G[V(D)]=D$. Suppose that $D$
is a tree. Since $D$ has an edge, it has at least two leaves, and since
every vertex of $G$ has degree at least $2$ by (1), each leaf of $D$ is
an end of an edge in
$\delta_G(V(D))=F$. As $|F|=2$, each leaf is an end of exactly one such
edge and therefore has degree $2$ in $G$. Thus $V(D)$ contains at least
two vertices of degree $2$ in $G$, contradicting the first assertion
applied to $S=V(D)$. This proves (6).
\end{proof}

\begin{lem}\label{lem:cyc-2-3-cut-property}
Let $k\ge 1$ be an integer, let $G$ be a connected cyclically
$k$-edge-connected graph, and let $F\subseteq E(G)$ be a cyclic
disconnecting set of size $k$. Then $G-F$ has exactly two components. In
particular, $F$ is an edge-cut. Moreover, if $G$ is subcubic, then $F$ is
a matching.
\end{lem}

\begin{proof}
Since $F$ is a cyclic disconnecting set, at least two components of $G-F$
contain cycles. Let $D$ be one such component. Since another component of
$G-F$ contains a cycle, the cut $\delta_G(V(D))$ is cyclic. Moreover,
$\delta_G(V(D))\subseteq F$, because every edge from $D$ to
$V(G)\setminus V(D)$ has been deleted by $F$. Since $G$ is cyclically
$k$-edge-connected,
\[
k\le |\delta_G(V(D))|\le |F|=k.
\]
Thus $\delta_G(V(D))=F$. In particular, every edge of $F$ has one end in
$D$ and one end outside $D$.

Let $D'$ be another component of $G-F$ containing a cycle. The same
argument gives $\delta_G(V(D'))=F$. Hence every edge of $F$ joins $D$ to
$D'$. Since $G$ is connected, there can be no third component of $G-F$.
Therefore $G-F$ has exactly two components, and $F$ is an edge-cut.

It remains to prove the matching statement under the additional
assumption that $G$ is subcubic. Let the two components of $G-F$ have
vertex sets $S$ and $V(G)\setminus S$, so that $F=\delta_G(S)$. Suppose,
for a contradiction, that two edges of $F$ have a common end $u\in S$.
Let $r=d_G(u,V(G)\setminus S)$. Then $r\ge 2$. Since $G$ is subcubic,
$d_G(u,S)\le 1$. Hence no cycle of $G[S]$ uses $u$.

Put $S'=S\setminus\{u\}$. Then
\[
|\delta_G(S')|
=
|\delta_G(S)|-r+d_G(u,S)
\le k-2+1
=k-1.
\]
Both sides of this new cut still contain cycles: the side $S'$ contains
a cycle because $G[S]$ contains a cycle and no cycle of $G[S]$ uses $u$,
while the other side already contained a cycle before $u$ was moved.
Thus $\delta_G(S')$ is a cyclic edge-cut of size less than $k$,
contradicting that $G$ is cyclically $k$-edge-connected.

The same argument applies if two edges of $F$ have a common end in
$V(G)\setminus S$. Therefore no two edges of $F$ share an end, and so
$F$ is a matching.
\end{proof}

We adopt the following terminology for the remainder of the paper. If
$G$ is a connected subcubic cyclically $3$-edge-connected graph, we call
every cyclic disconnecting set of size $3$ a \emph{cyclic $3$-edge-cut}. If $G$
is a $3$-critical graph, we call every essential disconnecting set of
size $2$ an \emph{essential $2$-edge-cut}; such a set is cyclic by
Lemma~\ref{lem:basic-3-critical}(6), and $G$ is cyclically
$2$-edge-connected because it is $2$-edge-connected by
Lemma~\ref{lem:basic-3-critical}(1). Hence, by
Lemma~\ref{lem:cyc-2-3-cut-property}, for every cyclic $3$-edge-cut and
every essential $2$-edge-cut $F$, the graph $G-F$ has exactly two
components and $F$ is a matching. If $F=\delta_G(S)$ is such a cut, we
call $S$ and $V(G)\setminus S$ the \emph{shores} of $F$; thus $G[S]$ and
$G[V(G)\setminus S]$ are the two components of $G-F$.

\subsection{Properties of the three operations}

We first recall a contraction lemma of Gr\"unewald and Steffen, which
gives the reverse directions of the vertex-blowup and of the
Meredith-type extension. Here $k$-critical multigraphs are defined in
the same way as $k$-critical graphs. For a multigraph $H$ and an integer
$k\ge\Delta(H)$, the \emph{$k$-deficiency} of $H$ is
\[
\df_k(H)=\sum_{v\in V(H)}(k-d_H(v)).
\]

\begin{lem}[{\cite[Lemma 2.3]{GruenewaldSteffen}}]\label{lem:contract-deficiency}
Let $G$ be a $k$-critical multigraph, where $k\ge 2$, and let $H$ be an
induced proper submultigraph of odd order. If $\df_k(H)=k$, then the multigraph
  obtained from $G$ by contracting $H$ to a single vertex $v$ is
$k$-critical.
\end{lem}

The next result shows that a vertex-blowup preserves $3$-criticality in
both directions. That $G$ is $3$-critical whenever $H$ is $3$-critical is
Lemma~3 of~\cite{BokalBrinkmannGrunewald2005}. Conversely, if $G$ is
$3$-critical, then the new triangle is an induced proper subgraph of $G$
of order $3$ with $3$-deficiency $3$, and contracting it gives $H$; hence
$H$ is $3$-critical by Lemma~\ref{lem:contract-deficiency} with $k=3$.

\begin{lem}\label{lem:blow-up}
Let $H$ be a subcubic  graph and $G$ a graph obtained from $H$
by  a vertex-blowup. Then $G$ is 3-critical if and only if $H$
is 3-critical.
\end{lem}

The following Haj\'os-join preservation result is classical; see
Jakobsen~\cite[Theorem~6]{Jakobsen1973}.
\begin{lem}\label{lem:Hajos1}
Let $G_1$ and $G_2$ be vertex-disjoint $3$-critical graphs, and let
$u_1v_1\in E(G_1)$ and $u_2v_2\in E(G_2)$ with
\[
d_{G_1}(u_1)+d_{G_2}(u_2)\le 5.
\]
Then the Haj\'os-join obtained from $G_1-u_1v_1$ and $G_2-u_2v_2$ by
identifying $u_1$ and $u_2$ and adding the edge $v_1v_2$ is
$3$-critical.
\end{lem}

A triangle $T$ in $G$ is \emph{independent} if no
vertex in $V(G)\setminus V(T)$ is adjacent to more than one vertex of
$T$. That the graph $G$ in Theorem~\ref{lem:Hajos2} below is a
Haj\'os-join of two $3$-critical graphs is Theorem~12.4 in the book of
Fiorini and Wilson~\cite{FioriniWilson1977}, where the result is stated
for simple $r$-critical graphs, without the assumption that there is no
independent triangle, and attributed to Jakobsen, who proved a version
for multigraphs~\cite[Theorem~7]{Jakobsen1973}. The proof of
Theorem~\ref{lem:Hajos2} follows exactly the argument in the proof
of~\cite[Theorem~12.4]{FioriniWilson1977}, and we therefore omit it. That at least one of the two graphs is nontrivial, which
is needed for Theorem~\ref{thm:characterization}(b), then follows, since
a Haj\'os-join of two graphs with minimum degree at least $2$ does not
decrease the degree of any vertex of either graph, including the two
vertices that are identified, and $3$-critical graphs have minimum degree
at least $2$ by Lemma~\ref{lem:basic-3-critical}(1). Hence $G$ has at
most as many vertices of degree $2$ as the two graphs together, and if
both graphs were trivial, then $G$ would have at most two vertices of
degree $2$, contrary to Lemma~\ref{lem:basic-3-critical}(5).

\begin{thm}\label{lem:Hajos2}
Let $G$ be a nontrivial $3$-critical graph with no independent triangle.
If $G$ has an essential $2$-edge-cut, then $G$ is the Haj\'os-join of two
$3$-critical graphs, at least one of which is nontrivial.
\end{thm}

In the remainder of this subsection, we study when a Meredith-type
extension with $\Delta=3$ is $3$-critical. The answer depends on how edge
$3$-colorings of the two graphs involved can color the three edges joining
them, which we record as follows.

Let $G$ be a subcubic graph with three distinct specified vertices
$x_1,x_2,x_3$, each of degree less than $3$. We denote by $G^p$ the graph
obtained from $G$ by adding three pendant edges $e_1,e_2,e_3$, where
$e_i$ is incident with $x_i$ for each $i\in[3]$. We say that a set
$F\subseteq \{e_1,e_2,e_3\}$ with $|F|=2$ is \emph{realized} in $G^p$ if
$G^p$ has a proper edge $3$-coloring in which the two edges in $F$
receive the same color and the remaining distinguished pendant edge
receives a different color. We say that $G^p$ has a
\emph{monochromatic state} if $G^p$ has a proper edge $3$-coloring in
which $e_1,e_2,e_3$ all receive the same color. More generally, the
\emph{boundary state} of a proper edge $3$-coloring of $G^p$ is the
partition of $\{e_1,e_2,e_3\}$ according to the colors received; it is
\emph{rainbow} if the three colors are distinct, a \emph{pair state} if
exactly two of the colors agree, and \emph{monochromatic} if all three
agree. Thus $F$ is realized in $G^p$ precisely when some coloring has the
pair state in which the two edges of $F$ share a color.

Let $D$ and $H$ be two subcubic graphs, and suppose that $D^p$ and $H^p$
have distinguished pendant edges $f_1,f_2,f_3$ and $e_1,e_2,e_3$,
respectively. Let
\[
\mathcal S_D=\{F\subseteq \{f_1,f_2,f_3\}: |F|=2 \text{ and } F
\text{ is realized in } D^p\}.
\]
For each edge $e\in E(H)$, let
\[
\mathcal S_H(e)=\{F\subseteq \{e_1,e_2,e_3\}: |F|=2 \text{ and } F
\text{ is realized in } H^p-e\}.
\]
We also say that $H^p-e$ has a monochromatic state if it has a proper
edge $3$-coloring in which $e_1,e_2,e_3$ all receive the same color.

We say that $D$ and $H$ are \emph{compatible} if one of the following
holds.
\begin{enumerate}
    \item[Type I:] $\mathcal S_D$ consists of all three pairs from
    $\{f_1,f_2,f_3\}$, and $\mathcal S_H(e)\neq \emptyset$ for every
    $e\in E(H)$.

    \item[Type II:] $|\mathcal S_D|=2$, and $|\mathcal S_H(e)|\ge 2$ for
    every $e\in E(H)$.

    \item[Type III:] $D^p$ has a monochromatic state, and $H^p-e$ has a
    monochromatic state for every $e\in E(H)$.
\end{enumerate}

The next theorem shows that, if $D$ is obtained from a nontrivial
$3$-critical graph by deleting a vertex of degree $3$, $H$ is a connected
class~$1$ graph of maximum degree $3$ with exactly three vertices of
degree $2$ and all other vertices of degree $3$, and $D$ and $H$ are
compatible, then every Meredith-type extension built from $D$ and $H$ is a
nontrivial $3$-critical graph.

\begin{thm}\label{thm:compatible-extension}
Let $G$ be a nontrivial $3$-critical graph and let $y\in V(G)$ with
$N_G(y)=\{y_1,y_2,y_3\}$. Let $H$ be a  connected class~$1$ graph of maximum degree
$3$ with exactly three vertices $x_1,x_2,x_3$ of degree $2$, and all other
vertices of degree $3$.

Let $D=G-y$,  and let $D^p$ be
obtained from $D$ by adding pendant edges $f_i$ incident with $y_i$ for
each $i\in[3]$. Let $H^p$ be obtained from $H$ by adding pendant edges
$e_i$ incident with $x_i$ for each $i\in[3]$. Suppose that $D$ and $H$
are compatible.
Then, for every bijection $\sigma:[3]\to[3]$, the graph
\[
J=(G-y)\cup H+\{y_i x_{\sigma(i)}: i\in[3]\}
\]
is a nontrivial $3$-critical graph.
\end{thm}

\begin{proof}
First note a simple consequence of the assumptions on $H$. Since $H$ is
class~$1$, $H$ has a proper edge $3$-coloring, and coloring each $e_i$
with the color missing at $x_i$ extends it to a proper edge $3$-coloring
of $H^p$. Since $H$ has exactly
three vertices of degree $2$ and all other vertices have degree $3$, we
have $|V(H)|$ odd. In any proper edge $3$-coloring of $H^p$, the colors
on $e_1,e_2,e_3$ are precisely the colors missing at
$x_1,x_2,x_3$ in the restriction to $H$. Hence, by the Parity Lemma
(Lemma~\ref{lem:parity-lemma}), in every proper edge $3$-coloring of
$H^p$, the pendant edges $e_1,e_2,e_3$ receive three distinct colors.

We first show that $J$ is not $3$-edge-colorable. Suppose, to the
contrary, that $J$ has a proper edge $3$-coloring. Restrict this coloring
to $H$ and view the three edges $y_i x_{\sigma(i)}$ as the pendant edges
$e_{\sigma(i)}$ of $H^p$. By the preceding paragraph, these three edges
receive three distinct colors. Now restrict the same coloring to $D=G-y$
and view the three edges $y_i x_{\sigma(i)}$ as the pendant edges
$f_i$ of $D^p$. Then $f_1,f_2,f_3$ receive three distinct colors, so we
can replace them by the three edges $yy_1,yy_2,yy_3$ and obtain a proper
edge $3$-coloring of $G$, a contradiction. Hence $J$ is not
$3$-edge-colorable.

It remains to show that deleting any edge of $J$ gives a
$3$-edge-colorable graph. We consider three cases.

First let $e\in E(G-y)$. Since $G$ is $3$-critical, $G-e$ has a proper
edge $3$-coloring. At the vertex $y$, the three edges
$yy_1,yy_2,yy_3$ receive three distinct colors. Deleting $y$ and
replacing these three edges by the pendant edges $f_1,f_2,f_3$ gives a
proper edge $3$-coloring of $D^p-e$ in which $f_1,f_2,f_3$ receive three
distinct colors. Also, as shown above, $H^p$ has a proper edge
$3$-coloring in which $e_1,e_2,e_3$ receive three distinct colors. After
a permutation of colors, these two colorings agree on the identified
pendant edges. Hence they combine to give a proper edge $3$-coloring of
$J-e$.

Next let $e=y_i x_{\sigma(i)}$ be one of the three new edges. Since $G$
is $3$-critical, $G-yy_i$ has a proper edge $3$-coloring. The two edges
$yy_j,yy_k$, where $\{i,j,k\}=[3]$, receive distinct colors. Deleting
$y$ and replacing $yy_j,yy_k$ by the pendant edges $f_j,f_k$ gives a
proper edge $3$-coloring of $D^p-f_i$ in which $f_j$ and $f_k$ receive
distinct colors. On the $H$-side, take a proper
edge $3$-coloring of $H^p$ in which $e_1,e_2,e_3$ receive three distinct
colors, delete the pendant edge $e_{\sigma(i)}$, and permute the colors
so that $e_{\sigma(j)}$ and $e_{\sigma(k)}$ receive the colors of $f_j$
and $f_k$, respectively. These two colorings combine to give a proper edge $3$-coloring of
$J-e$.

Now let $e\in E(H)$. We use compatibility.

If Type I holds, then $\mathcal S_H(e)\neq\emptyset$; choose a pair
$Q_0\in \mathcal S_H(e)$. Since $\mathcal S_D$ contains all three pairs,
the corresponding pair $P=\{f_i:e_{\sigma(i)}\in Q_0\}$
is realized in $D^p$.

If Type II holds, then $|\mathcal S_D|=2$ and $|\mathcal S_H(e)|\ge 2$.
Under the bijection $\sigma$, the two pairs in $\mathcal S_D$ correspond
to two pairs of $\{e_1,e_2,e_3\}$. Since two families of two $2$-subsets
of a $3$-element set must intersect, there is a pair
$Q_0\in \mathcal S_H(e)$ whose corresponding pair $P=\{f_i:e_{\sigma(i)}\in Q_0\}$
lies in $\mathcal S_D$.

In either Type I or Type II, choose a coloring of $D^p$ realizing $P$ and
a coloring of $H^p-e$ realizing $Q_0$. After a permutation of colors on
one side, the colors on the identified pendant edges agree. Hence the two
colorings combine to give a proper edge $3$-coloring of $J-e$.

It remains to consider Type III. In this case, both $D^p$ and $H^p-e$
have monochromatic states. Choose such colorings. After a permutation of
colors on one side, all three pendant edges on both sides receive the
same color. Hence the two colorings agree on the identified pendant
edges, and they combine to give a proper edge $3$-coloring of $J-e$.

Finally, $J$ is connected, since $H$ is connected and every component of
$G-y$ contains one of $y_1,y_2,y_3$, each of which is adjacent in $J$ to a
vertex of $H$. Every vertex of $G-y$ has the same degree in $J$ as in $G$,
and each $x_i$ has degree $3$ in $J$. Hence $\Delta(J)=3$, and $J$ is
$3$-critical. Moreover, all degree-$2$ vertices of $G$ remain degree-$2$
vertices of $J$. Since $G$ is nontrivial, $J$ has at least three vertices
of degree $2$, and $J$ is nontrivial by
Lemma~\ref{lem:basic-3-critical}(5).
\end{proof}

We now illustrate Theorem~\ref{thm:compatible-extension} by one
compatible pair of each type. The graphs $G$ and $H$ of each example are
shown, with the vertex labels used below, in
Figures~\ref{fig:compatible-type-I-example}--\ref{fig:compatible-type-III-example}.
In each example,
the sets $\mathcal S_D$ and $\mathcal S_H(e)$ and the monochromatic states
were determined by checking all edge $3$-colorings of the small graphs
involved, and $J$ denotes the extension of
Theorem~\ref{thm:compatible-extension} for the identity bijection
$\sigma$.

\paragraph{Type I example.}

Let $G$ be the unique nontrivial $3$-critical graph of order $9$, the
Petersen graph minus a vertex, labeled as in
Figure~\ref{fig:compatible-type-I-example}. Let $y=1$, so that
$N_G(y)=\{4,6,8\}$, and
let the pendant edges $f_1,f_2,f_3$ of $D^p$ be incident with $4,6,8$,
respectively. Then
\[
  \mathcal S_D=\big\{\{f_1,f_2\},\{f_1,f_3\},\{f_2,f_3\}\big\}.
\]
Let $H$ be the class~$1$ graph of order $7$ in
Figure~\ref{fig:compatible-type-I-example}, whose three degree-$2$
vertices are $x_1=0$, $x_2=1$, and $x_3=2$. Then $\mathcal S_H(e)\neq\emptyset$ for
every $e\in E(H)$, so $D$ and $H$ are compatible of Type~I. Here
$|\mathcal S_H(e)|=1$ for exactly one edge, namely $46$, where
\[
  \mathcal S_H(46)=\big\{\{e_1,e_3\}\big\},
\]
whereas for $H=K_{2,3}$ one has $|\mathcal S_H(e)|=3$ for every edge $e$.
By Theorem~\ref{thm:compatible-extension}, the extension $J$ is a
nontrivial $3$-critical graph of order $15$.

\begin{figure}[htbp]
  \centering
  \begin{tikzpicture}[scale=1.0,
      line cap=round,line join=round,
      vtx/.style={circle,draw,fill=white,inner sep=1.2pt,minimum size=5mm},
      yvtx/.style={circle,draw,fill=red!70,inner sep=1.2pt,minimum size=5mm},
      nbr/.style={circle,draw,fill=blue!35,inner sep=1.2pt,minimum size=5mm},
      bdry/.style={circle,draw,fill=orange!75,inner sep=1.2pt,minimum size=5mm}]

    % Left graph G
    \node[vtx]  (g0) at (0.000,1.600) {\scriptsize $0$};
    \node[yvtx] (g1) at (1.028,1.226) {\scriptsize $1$};
    \node[vtx]  (g2) at (1.576,0.278) {\scriptsize $2$};
    \node[vtx]  (g3) at (1.386,-0.800) {\scriptsize $3$};
    \node[nbr]  (g4) at (0.547,-1.504) {\scriptsize $4$};
    \node[vtx]  (g5) at (-0.547,-1.504) {\scriptsize $5$};
    \node[nbr]  (g6) at (-1.386,-0.800) {\scriptsize $6$};
    \node[vtx]  (g7) at (-1.576,0.278) {\scriptsize $7$};
    \node[nbr]  (g8) at (-1.028,1.226) {\scriptsize $8$};

    \draw (g0) -- (g3);
    \draw (g0) -- (g6);
    \draw (g1) -- (g4);
    \draw (g1) -- (g6);
    \draw (g1) -- (g8);
    \draw (g2) -- (g5);
    \draw (g2) -- (g6);
    \draw (g2) -- (g7);
    \draw (g3) -- (g7);
    \draw (g3) -- (g8);
    \draw (g4) -- (g7);
    \draw (g5) -- (g8);

    \node at (0.000,-2.050) {$G$};

    % Right graph H --- improved layout
    \node[bdry] (h0) at (4.90,-1.65) {\scriptsize $0$};
    \node[bdry] (h1) at (6.05, 1.55) {\scriptsize $1$};
    \node[vtx]  (h6) at (7.55,-1.65) {\scriptsize $6$};

    \node[vtx]  (h5) at (5.38,-0.15) {\scriptsize $5$};
    \node[vtx]  (h3) at (6.72, 0.30) {\scriptsize $3$};
    \node[vtx]  (h4) at (6.12,-1.02) {\scriptsize $4$};
    \node[bdry] (h2) at (6.18,-0.35) {\scriptsize $2$};

    \draw (h0) -- (h5);
    \draw (h0) -- (h6);
    \draw (h1) -- (h3);
    \draw (h1) -- (h5);
    \draw (h2) -- (h3);
    \draw (h2) -- (h4);
    \draw (h3) -- (h6);
    \draw (h4) -- (h5);
    \draw (h4) -- (h6);

    \node at (6.20,-2.18) {$H$};
  \end{tikzpicture}
  \caption{A compatible pair of Type~I. Left: the order-$9$ graph $G$ with
  $y=1$ (red) and $N_G(y)=\{4,6,8\}$ (blue). Right: the class~$1$ graph $H$
  with the three degree-$2$ vertices $x_1=0,x_2=1,x_3=2$ (orange).}
  \label{fig:compatible-type-I-example}
\end{figure}

\paragraph{Type II example.}

Let $G$ be the nontrivial $3$-critical graph of order $11$ shown in
Figure~\ref{fig:compatible-type-II-example}. Let $y=2$, so that
$N_G(y)=\{6,9,10\}$, and let the pendant edges $f_1,f_2,f_3$ of $D^p$ be
incident with $6,9,10$, respectively. Then
\[
  \mathcal S_D=\big\{\{f_1,f_2\},\{f_1,f_3\}\big\},
\]
so the pair $\{f_2,f_3\}$ is not realized; moreover, $D^p$ has no
monochromatic state. Let $H$ be the class~$1$ graph of order $5$ in
Figure~\ref{fig:compatible-type-II-example}, whose three degree-$2$
vertices are $x_1=0$, $x_2=1$, and $x_3=2$. Then
$|\mathcal S_H(e)|\ge 2$ for every $e\in E(H)$, with equality, for
instance, at the edge $02$, where
\[
  \mathcal S_H(02)=\big\{\{e_1,e_2\},\{e_2,e_3\}\big\}.
\]
Hence $D$ and $H$ are compatible of Type~II, and by
Theorem~\ref{thm:compatible-extension} the extension $J$ is a nontrivial
$3$-critical graph of order $15$.

By Proposition~\ref{prop:kotzig} below, the graph $H$ in
Theorem~\ref{thm:compatible-extension} may be taken to be any graph
obtained from a Kotzig graph other than $K_4$ by deleting a vertex,
whatever $G$ and $y$ are. Here a cubic graph is called a
\emph{Kotzig graph} if it has a proper edge $3$-coloring in which the
edges of any two colors form a Hamiltonian cycle. Every cubic graph whose
proper edge $3$-coloring is unique up to a permutation of the colors is
a Kotzig graph. Indeed, the edges of any two colors form a spanning
$2$-regular subgraph, that is, a disjoint union of cycles; if there were
at least two such cycles, then interchanging the two colors on one of
them would give a proper edge $3$-coloring that agrees with the original
one on the edges of the third color and on another cycle containing both
colors, and hence is not obtained from it by permuting the colors. For
example, $K_4$, the triangular prism, and the generalized Petersen graph
$P(9,2)$ (see~\cite{BelcastroHaas2015}) are uniquely $3$-edge-colorable
and hence Kotzig graphs, and
the graph $H$ of the Type~II example is the triangular prism minus a
vertex. Also $K_{3,3}$ is a Kotzig graph, since any two disjoint perfect
matchings of $K_{3,3}$ form a Hamiltonian cycle, and $K_{3,3}$ minus a
vertex is $K_{2,3}$. There are infinitely many Kotzig graphs of girth at
least $4$: Belcastro and Haas~\cite{BelcastroHaas2015} constructed
infinitely many triangle-free uniquely $3$-edge-colorable cubic graphs,
and each of them is a Kotzig graph by the argument above. (For $K_4$ the graph $H$ is a triangle,
which has maximum degree $2$, and the corresponding extension is simply a
vertex-blowup.)

\begin{prop}\label{prop:kotzig}
Let $K$ be a Kotzig graph other than $K_4$, let $v\in V(K)$ with
$N_K(v)=\{x_1,x_2,x_3\}$, and let $H=K-v$. Let $e_i=x_iy_i$, $i\in[3]$,
be the three pendant edges of $H^p$. Then $H$ is a connected class~$1$
graph of maximum degree $3$ whose vertices of degree $2$ are exactly
$x_1,x_2,x_3$, all other vertices having degree $3$, and
$|\mathcal S_H(e)|\ge 2$ for every $e\in E(H)$. In particular, if $G$ is
a nontrivial $3$-critical graph and $y\in V(G)$ has degree $3$, then
$D=G-y$ and $H$ are compatible, of Type~I if $|\mathcal S_D|=3$ and of
Type~II otherwise.
\end{prop}

\begin{proof}
Since $K$ is cubic and $K\ne K_4$, the graph $K$ has at least six
vertices, so $H$ has maximum degree $3$ and its vertices of degree $2$
are exactly $x_1,x_2,x_3$. Moreover, $H$ is of class~$1$, since it
inherits a proper edge $3$-coloring from $K$. Note that identifying
$y_1,y_2,y_3$ in $H^p$ into a single vertex gives $K$.

Let $e\in E(H)$, and let $\varphi$ be the proper edge $3$-coloring of
$H^p$ obtained from a proper edge $3$-coloring of $K$ in which the edges
of any two colors form a Hamiltonian cycle; thus
$\varphi(x_1y_1),\varphi(x_2y_2),\varphi(x_3y_3)$ are the colors of the
three edges of $K$ at $v$, which are distinct. Since the condition
$|\mathcal S_H(e)|\ge 2$ is symmetric in the indices, by renaming the
colors and the indices we may assume that $\varphi(x_iy_i)=i$ for
$i\in[3]$ and $\varphi(e)=1$.

For $j\in\{2,3\}$, the
Hamiltonian cycle of $K$ formed by the colors $1$ and $j$ passes through
$v$ along the edges $vx_1$ and $vx_j$. Consequently, the edges of $H^p$
with colors $1$ and $j$ form a path $P_j$ joining $y_1$ and $y_j$ that
contains all vertices of $H$. In particular, $H$ is connected, and $P_j$
contains every edge of color $1$ of $H^p$, including $e$.

Now fix $j\in\{2,3\}$, and let $k$ be the third index, so that
$\{1,j,k\}=[3]$. Since $e$ is an edge of $P_j$ other than $x_1y_1$ and
$x_jy_j$, the graph $P_j-e$ consists of two paths: $Q_1$, which contains
$y_1$ and the edge $x_1y_1$, and $Q_j$, which contains $y_j$ and the edge
$x_jy_j$. In $H^p-e$, the edges of colors $1$ and $j$ form $P_j-e$, so
$Q_j$ is a component of the subgraph formed by these edges.
Interchanging the colors $1$ and $j$ on $Q_j$ therefore turns the
restriction of $\varphi$ to $H^p-e$ into a proper edge $3$-coloring of
$H^p-e$ in which $x_1y_1$ and $x_jy_j$ both have color $1$, while the
color $k$ of $x_ky_k$ is unchanged. Hence $\{e_1,e_j\}$ is realized in
$H^p-e$. As this holds for both $j=2$ and $j=3$, we obtain
$|\mathcal S_H(e)|\ge 2$.

For the last assertion, the proof of
Theorem~\ref{lem:cyclic-cut-reduction}(4) below uses only that $G$ is
$3$-critical and that $y$ has degree $3$, and it shows that
$|\mathcal S_D|\ge 2$. Since $|\mathcal S_H(e)|\ge 2$ for every
$e\in E(H)$, the pair $D$ and $H$ is compatible of Type~I if
$|\mathcal S_D|=3$ and of Type~II if $|\mathcal S_D|=2$.
\end{proof}

\begin{figure}[htbp]
  \centering
  \begin{tikzpicture}[scale=1.0,
      vtx/.style={circle,draw,fill=white,inner sep=1.2pt,minimum size=5mm},
      yvtx/.style={circle,draw,fill=red!70,inner sep=1.2pt,minimum size=5mm},
      nbr/.style={circle,draw,fill=blue!35,inner sep=1.2pt,minimum size=5mm},
      bdry/.style={circle,draw,fill=orange!75,inner sep=1.2pt,minimum size=5mm}]
    \node[vtx] (g0) at (0.000,1.950) {\scriptsize $0$};
    \node[vtx] (g1) at (0.865,1.346) {\scriptsize $1$};
    \node[yvtx] (g2) at (1.455,0.665) {\scriptsize $2$};
    \node[vtx] (g3) at (1.584,-0.228) {\scriptsize $3$};
    \node[vtx] (g4) at (1.209,-1.048) {\scriptsize $4$};
    \node[vtx] (g5) at (0.451,-1.535) {\scriptsize $5$};
    \node[nbr] (g6) at (-0.451,-1.535) {\scriptsize $6$};
    \node[vtx] (g7) at (-1.209,-1.048) {\scriptsize $7$};
    \node[vtx] (g8) at (-1.584,-0.228) {\scriptsize $8$};
    \node[nbr] (g9) at (-1.455,0.665) {\scriptsize $9$};
    \node[nbr] (g10) at (-0.865,1.346) {\scriptsize $10$};
    \draw (g0) -- (g4);
    \draw (g0) -- (g7);
    \draw (g1) -- (g5);
    \draw (g1) -- (g7);
    \draw (g1) -- (g10);
    \draw (g2) -- (g6);
    \draw (g2) -- (g9);
    \draw (g2) -- (g10);
    \draw (g3) -- (g7);
    \draw (g3) -- (g8);
    \draw (g3) -- (g9);
    \draw (g4) -- (g8);
    \draw (g4) -- (g10);
    \draw (g5) -- (g8);
    \draw (g6) -- (g9);
    \node at (0.000,-2.050) {$G$};
    \node[bdry] (h0) at (4.400,-1.200) {\scriptsize $0$};
    \node[bdry] (h1) at (5.800,1.200) {\scriptsize $1$};
    \node[bdry] (h2) at (4.800,1.200) {\scriptsize $2$};
    \node[vtx] (h3) at (5.300,0.000) {\scriptsize $3$};
    \node[vtx] (h4) at (6.200,-1.200) {\scriptsize $4$};
    \draw (h0) -- (h2);
    \draw (h0) -- (h4);
    \draw (h1) -- (h3);
    \draw (h1) -- (h4);
    \draw (h2) -- (h3);
    \draw (h3) -- (h4);
    \node at (5.300,-2.050) {$H$};
  \end{tikzpicture}
  \caption{A compatible pair of Type~II. Left: the order-$11$ graph $G$
  with $y=2$ (red) and $N_G(y)=\{6,9,10\}$ (blue). Right: the class~$1$
  graph $H$ with the three degree-$2$ vertices $x_1=0,x_2=1,x_3=2$
  (orange).}
  \label{fig:compatible-type-II-example}
\end{figure}

\paragraph{Type III example.}

Let $G$ be the nontrivial $3$-critical graph of order $13$ shown in
Figure~\ref{fig:compatible-type-III-example}. Let $y=0$, so that
$N_G(y)=\{5,8,9\}$, and let the pendant edges $f_1,f_2,f_3$ of $D^p$ be
incident with $5,8,9$, respectively. Then
\[
  \mathcal S_D=\big\{\{f_1,f_2\},\{f_1,f_3\}\big\},
\]
and $D^p$ has a monochromatic state. Let $H$ be the class~$1$ graph of
order $9$ in Figure~\ref{fig:compatible-type-III-example}, whose three
degree-$2$ vertices are $x_1=0$, $x_2=1$, and $x_3=2$. Then $H^p-e$ has a
monochromatic state for every $e\in E(H)$, so $D$ and $H$ are compatible
of Type~III. They are not compatible of Type~I or of Type~II, since the
pair $\{f_2,f_3\}$ is not realized in $D^p$ and
\[
  \mathcal S_H(46)=\big\{\{e_1,e_2\}\big\}.
\]
By Theorem~\ref{thm:compatible-extension}, the extension $J$ is a
nontrivial $3$-critical graph of order $21$.

\begin{figure}[htbp]
  \centering
  \begin{tikzpicture}[scale=1.0,
      vtx/.style={circle,draw,fill=white,inner sep=1.2pt,minimum size=5mm},
      yvtx/.style={circle,draw,fill=red!70,inner sep=1.2pt,minimum size=5mm},
      nbr/.style={circle,draw,fill=blue!35,inner sep=1.2pt,minimum size=5mm},
      bdry/.style={circle,draw,fill=orange!75,inner sep=1.2pt,minimum size=5mm}]
    \node[yvtx] (g0) at (0.000,2.050) {\scriptsize $0$};
    \node[vtx] (g1) at (1.359,-1.534) {\scriptsize $1$};
    \node[vtx] (g2) at (-0.491,-1.990) {\scriptsize $2$};
    \node[vtx] (g3) at (1.687,1.165) {\scriptsize $3$};
    \node[vtx] (g4) at (-1.687,1.165) {\scriptsize $4$};
    \node[nbr] (g5) at (-2.035,0.247) {\scriptsize $5$};
    \node[vtx] (g6) at (0.491,-1.990) {\scriptsize $6$};
    \node[vtx] (g7) at (-1.359,-1.534) {\scriptsize $7$};
    \node[nbr] (g8) at (-0.953,1.815) {\scriptsize $8$};
    \node[nbr] (g9) at (0.953,1.815) {\scriptsize $9$};
    \node[vtx] (g10) at (1.917,-0.727) {\scriptsize $10$};
    \node[vtx] (g11) at (-1.917,-0.727) {\scriptsize $11$};
    \node[vtx] (g12) at (2.035,0.247) {\scriptsize $12$};
    \draw (g0) -- (g5);
    \draw (g0) -- (g8);
    \draw (g0) -- (g9);
    \draw (g1) -- (g6);
    \draw (g1) -- (g10);
    \draw (g2) -- (g7);
    \draw (g2) -- (g10);
    \draw (g2) -- (g12);
    \draw (g3) -- (g8);
    \draw (g3) -- (g9);
    \draw (g3) -- (g12);
    \draw (g4) -- (g8);
    \draw (g4) -- (g9);
    \draw (g5) -- (g10);
    \draw (g5) -- (g11);
    \draw (g6) -- (g11);
    \draw (g6) -- (g12);
    \draw (g7) -- (g11);
    \node at (0.000,-2.600) {$G$};
    \node[bdry] (h0) at (6.000,1.650) {\scriptsize $0$};
    \node[bdry] (h1) at (7.061,1.264) {\scriptsize $1$};
    \node[bdry] (h2) at (6.564,-1.550) {\scriptsize $2$};
    \node[vtx] (h3) at (7.429,-0.825) {\scriptsize $3$};
    \node[vtx] (h4) at (4.571,-0.825) {\scriptsize $4$};
    \node[vtx] (h5) at (7.625,0.287) {\scriptsize $5$};
    \node[vtx] (h6) at (4.375,0.287) {\scriptsize $6$};
    \node[vtx] (h7) at (5.436,-1.550) {\scriptsize $7$};
    \node[vtx] (h8) at (4.939,1.264) {\scriptsize $8$};
    \draw (h0) -- (h1);
    \draw (h0) -- (h8);
    \draw (h1) -- (h5);
    \draw (h2) -- (h3);
    \draw (h2) -- (h7);
    \draw (h3) -- (h5);
    \draw (h3) -- (h6);
    \draw (h4) -- (h5);
    \draw (h4) -- (h6);
    \draw (h4) -- (h7);
    \draw (h6) -- (h8);
    \draw (h7) -- (h8);
    \node at (6.000,-2.600) {$H$};
  \end{tikzpicture}
  \caption{A compatible pair of Type~III. Left: the order-$13$ graph $G$
  with $y=0$ (red) and $N_G(y)=\{5,8,9\}$ (blue). Right: the class~$1$
  graph $H$ with the three degree-$2$ vertices $x_1=0,x_2=1,x_3=2$
  (orange).}
  \label{fig:compatible-type-III-example}
\end{figure}

The next theorem is a partial converse of
Theorem~\ref{thm:compatible-extension}. Parts (1)--(3) show that a cyclic
$3$-edge-cut with a shore that is not a triangle and contains no vertex of
degree $2$ always comes from a Meredith-type extension; this gives case (c) of
Theorem~\ref{thm:characterization}. Parts (4) and (5) give necessary
conditions on the boundary colorings that are weaker than compatibility.

\begin{thm}\label{lem:cyclic-cut-reduction}
Let $G'$ be a cyclically $3$-edge-connected nontrivial $3$-critical graph,
and let $F=\{y_i x_i:i\in[3]\}$
be a cyclic $3$-edge-cut of $G'$. Let $H$ be one component of $G'-F$ with
$x_1,x_2,x_3\in V(H)$, and suppose that $H$ is not a triangle and
contains no vertex that has degree $2$ in $G'$. Let $G$ be obtained from $G'$ by contracting $H$ to a
single vertex $y$. Let $D=G-y$, with distinguished vertices
$y_1,y_2,y_3$, and let $D^p$ be obtained from $D$ by adding pendant edges
$f_i$ incident with $y_i$ for each $i\in[3]$. Let $H^p$ be obtained from
$H$ by adding pendant edges $e_i$ incident with $x_i$ for each $i\in[3]$.

Then the following hold.
\begin{enumerate}[(1)]

 \item $H$ is a connected class~$1$ graph with exactly three vertices
    of degree $2$, namely $x_1,x_2,x_3$, and all other vertices of $H$
    have degree $3$.

    \item $G$ is a nontrivial $3$-critical graph.

    \item $G'$ is the Meredith-type extension of $G$ through $H$ at $y$.

    \item $D^p$ realizes at least two pair states. That is, for at least
    two pairs $\{i,j\}\subseteq[3]$, the graph $D^p$ has a proper edge
    $3$-coloring in which $f_i$ and $f_j$ receive the same color and the
    remaining pendant edge receives a different color.

    \item For every edge $e\in E(H)$, either $D^p$ and $H^p-e$ have a
    common pair state, or both $D^p$ and $H^p-e$ have a
    monochromatic state.
\end{enumerate}
\end{thm}

\begin{proof}
By Lemma~\ref{lem:cyc-2-3-cut-property}, $G'-F$ has exactly two components
and $F$ is a matching in $G'$.

We  prove the degree assertion for $H$. Since $H$ contains no vertex
that has degree $2$ in $G'$, and since $G'$ is subcubic and $3$-critical,
every vertex of $H$ has degree $3$ in $G'$. Since
$F$ is a matching in $G'$,
each vertex $x_i$ is incident
with exactly one edge of $F$. Thus  $d_H(x_i)=2$ for each $i\in[3]$. Every
other vertex of $H$ is incident with no edge of $F$, and therefore has
degree $3$ in $H$. Thus $H$ has exactly three vertices of degree $2$,
namely $x_1,x_2,x_3$, and all other vertices of $H$ have degree $3$.
Since $H$ is not a triangle, it has a vertex other than $x_1,x_2,x_3$,
and so $\Delta(H)=3$.
Also, $H$ is connected because it is a component of $G'-F$.

We show that $H$ is class~$1$. Choose an edge $g\in F$. Since $G'$ is
$3$-critical, $G'-g$ has a proper edge $3$-coloring. Restricting this
coloring to $H$ gives a proper edge $3$-coloring of $H$. Since
$\Delta(H)=3$, $H$ is class~$1$. This proves (1).

We now prove that $G$ is a nontrivial $3$-critical graph. The subgraph
$H$ is an induced proper subgraph of $G'$, since it is a component of
$G'-F$. Moreover, $H$ has odd order, since it has exactly three vertices of
even degree and an even number of vertices of odd degree, and
$\df_3(H)=\sum_{v\in V(H)}(3-d_H(v))=3$.
Thus Lemma~\ref{lem:contract-deficiency} applies with $k=3$, and the
multigraph $G$ obtained by contracting $H$ is $3$-critical. Since $F$ is
a matching, the neighbors $y_1,y_2,y_3$ of the new vertex $y$ are
distinct, so $G$ is simple.

It remains to check that $G$ is nontrivial. Since $G'$ is nontrivial,
$G'$ has at least three vertices of degree $2$. By assumption, no vertex
of $H$ has degree $2$ in $G'$. Hence all degree-$2$ vertices of $G'$ lie
outside $H$, and their degrees are unchanged when $H$ is contracted. Thus
$G$ has at least three vertices of degree $2$, and so $G$ is nontrivial by
Lemma~\ref{lem:basic-3-critical}(5).
This proves (2).

By the definition of $G$, the vertex $y$ has neighbors $y_1,y_2,y_3$ in
$G$. Moreover, $G'$ is obtained from $G-y$ and $H$ by adding the three
edges $y_i x_i$ for $i\in[3]$. Thus $G'$ is the Meredith-type extension
of $G$ through $H$ at $y$. This proves (3).

Before proving (4) and (5), we note that $D^p$ has no rainbow boundary
state. Indeed, if
$D^p$ had a proper edge $3$-coloring in which $f_1,f_2,f_3$ received
three distinct colors, then identifying the three pendant ends into a
single vertex would give a proper edge $3$-coloring of $G$, contradicting
that $G$ is class~$2$.

For (4), fix $i\in[3]$, and let $\{i,j,k\}=[3]$. Since $G$ is
$3$-critical, $G-yy_i$ has a proper edge $3$-coloring. The two edges
$yy_j$ and $yy_k$ receive distinct colors. Deleting $y$ and replacing
$yy_j,yy_k$ by the pendant edges $f_j,f_k$, and then coloring $f_i$ with
a color missing at $y_i$, gives a proper edge $3$-coloring of $D^p$.
Since $f_j$ and $f_k$ receive distinct colors, and since $D^p$ has no
rainbow boundary state, the color on $f_i$ must agree with exactly one of
the colors on $f_j$ and $f_k$. Hence $D^p$ realizes a pair state involving
$f_i$. Since this holds for every $i\in[3]$ and a single pair state
involves only two of $f_1,f_2,f_3$, the graph $D^p$ realizes at least two
pair states. This proves (4).

For (5), let $e\in E(H)$. Since $G'$ is $3$-critical, $G'-e$ has a proper
edge $3$-coloring. Restricting this coloring to the two sides of the cut
gives a proper edge $3$-coloring of $D^p$ and a proper edge
$3$-coloring of $H^p-e$. These two colorings have the same boundary state
on the three cut edges, with $f_i$ corresponding to $e_i$ for each
$i\in[3]$. This common state cannot be a rainbow state, since $D^p$ has
no rainbow boundary state. Therefore the common state is either a pair
state or a monochromatic state. In the first case, $D^p$ and $H^p-e$ have
a common pair state. In the second case, both $D^p$ and $H^p-e$ have a
monochromatic state. This proves (5).
\end{proof}
\section{Proof of Theorem~\ref{thm:characterization}}\label{S:characterization}

In this section we prove Theorem~\ref{thm:characterization}. The proof
applies one of the reductions of Section~\ref{S:operations} whenever one
is available. When none of them applies, the following lemma provides the snark in cases (d) and (e); we
prove it in Subsection~\ref{SS:snark-proof}.

\begin{lem}\label{lem:snark-reduction}
Let $G$ be a cyclically $3$-edge-connected nontrivial $3$-critical graph of
girth at least $4$. Assume that for every cyclic $3$-edge-cut $F$ of $G$,
each component of $G-F$ contains a vertex of degree $2$ in $G$. Then $G$
is a subgraph of a snark $\widehat G$ that is obtained from $G$ as
follows, where a \emph{pairing} of a set is a partition of the set into
$2$-element subsets.
\begin{enumerate}[(1)]
    \item If $|V(G)|$ is even, then $\widehat G$ is obtained by adding,
    for some pairing of all degree-$2$ vertices of $G$, an edge between
    the two vertices of each pair.

    \item If $|V(G)|$ is odd, then $\widehat G$ is obtained by adding a
    new vertex $x$ adjacent to three degree-$2$ vertices $z_1,z_2,z_3$ of $G$
    and, for some pairing of the remaining degree-$2$ vertices of $G$, an
    edge between the two vertices of each pair.
\end{enumerate}
\end{lem}

\begin{proof}[Proof of Theorem~\ref{thm:characterization}]
Let $G$ be a nontrivial $3$-critical graph of
order $n$. By Lemma~\ref{lem:basic-3-critical}, $G$ is $2$-connected and
hence $2$-edge-connected, every vertex of $G$ has degree $2$ or $3$, and
$G$ has at least three vertices of degree $2$. We consider four cases; in each case we assume
that the preceding cases do not apply.

\medskip
\noindent\textbf{Case 1: $G$ has an independent triangle $T$.}
Let $G'$ be the graph obtained from $G$ by contracting $T$ to a single
vertex $x$. Each vertex of $T$ has at most one neighbor outside $T$,
since $G$ is subcubic, and no vertex outside $T$ has two neighbors in
$T$, since $T$ is independent. Hence $G'$ is a simple subcubic graph with
$|V(G')|=n-2$, and $G$ is the vertex-blowup of $G'$ at $x$. By
Lemma~\ref{lem:blow-up}, $G'$ is $3$-critical. By
Lemma~\ref{lem:basic-3-critical}(3), no two adjacent vertices of $G$ both
have degree $2$, so $T$ contains at most
one vertex of degree $2$ in $G$. If $T$ contains such a vertex, then $x$
has degree $2$ in $G'$; otherwise $x$ has degree $3$ in $G'$. Hence $G$
and $G'$ have the same number of vertices of degree $2$, and $G'$ is
nontrivial by Lemma~\ref{lem:basic-3-critical}(5). Thus (a) holds.

\medskip
\noindent\textbf{Case 2: $G$ has an essential $2$-edge-cut.}
Since Case~1 does not apply, $G$ has no independent triangle. By
Theorem~\ref{lem:Hajos2}, $G$ is a Haj\'os-join of two $3$-critical
graphs, at least one of which is nontrivial. A Haj\'os-join of graphs of
orders $n_1$ and $n_2$ has order $n_1+n_2-1$, and every $3$-critical graph
has at least two vertices, so both graphs have fewer than $n$ vertices.
Thus (b) holds.

\medskip
We now show that if Cases~1 and~2 do not apply, then $G$ has girth at
least $4$ and is cyclically $3$-edge-connected. Suppose that $G$ contains
a triangle $xyzx$. Since this triangle is not independent, some vertex
$w\in V(G)\setminus\{x,y,z\}$ is adjacent to at least two of $x,y,z$. If
$w$ were adjacent to all three, then $x,y,z,w$ would have degree $3$ in
$G[\{x,y,z,w\}]$, so $G=K_4$ because $G$ is connected; but $K_4$ is of
class~$1$. Hence $w$ is adjacent to exactly two of $x,y,z$, say to $x$
and $y$. Let
\[
    F=\delta_G(\{x,y,z,w\}).
\]
The vertices $x$ and $y$ have degree $3$ in $G[\{x,y,z,w\}]$, so only $z$
and $w$ can have neighbors outside $\{x,y,z,w\}$, at most one each. Hence
$|F|\le 2$. Since $G$ has at least three vertices of degree $2$, the set
$V(G)\setminus\{x,y,z,w\}$ is nonempty, and since $G$ is connected and
$2$-edge-connected, $|F|=2$. Consequently $z$ and $w$ have degree $3$, so
all vertices of degree $2$ lie outside $\{x,y,z,w\}$. Since $G$ is
$2$-edge-connected and $|F|=2$, $G-F$ has exactly two components (a
component of $G-\{x,y,z,w\}$ meeting only one edge of $F$ would give a
bridge): $G[\{x,y,z,w\}]$, which
contains a cycle, and $G-\{x,y,z,w\}$, which has at least three vertices
and hence contains an edge. Therefore $F$ is an essential $2$-edge-cut,
contradicting that Case~2 does not apply. Hence $G$ has girth at least
$4$. Moreover, $G$ has no cyclic disconnecting set with at most one edge,
because $G$ is $2$-edge-connected, and every cyclic disconnecting set with
two edges is an essential $2$-edge-cut. Hence $G$ is cyclically
$3$-edge-connected.

\medskip
\noindent\textbf{Case 3: $G$ has a cyclic $3$-edge-cut $F$ such that one
component $H$ of $G-F$ contains no vertex of degree $2$ in $G$.}
By Lemma~\ref{lem:cyc-2-3-cut-property}, $F$ is a matching, so we may
write $F=\{y_ix_i:i\in[3]\}$ with $x_1,x_2,x_3\in V(H)$. Since $G$ has
girth at least $4$, $H$ is not a triangle. Let $G'$ be the
graph obtained from $G$ by contracting $H$ to a single vertex. By
Theorem~\ref{lem:cyclic-cut-reduction}, applied with the roles of $G$
and $G'$ interchanged, the graph $G'$ is a nontrivial $3$-critical graph, $H$ is a class~$1$
graph with exactly three vertices of degree $2$, and $G$ is the
Meredith-type extension of $G'$ through $H$ at the contracted vertex.
Since $|V(H)|\ge 2$, the graph $G'$ has fewer than $n$ vertices. Thus (c)
holds.

\medskip
\noindent\textbf{Case 4: none of Cases 1--3 applies.}
Then $G$ has girth at least $4$, is cyclically $3$-edge-connected, and
for every cyclic $3$-edge-cut $F$ of $G$, each component of $G-F$
contains a vertex of degree $2$ in $G$. Let $\widehat G$ be the snark
given by Lemma~\ref{lem:snark-reduction}. The added edges joining two
vertices of $G$ correspond to the pairs of a pairing, so every vertex of
$G$ is an end of at most one of them; hence these edges are independent.
If $n$ is odd, then $\widehat G$ has order $n+1$, and $G$ is obtained from
$\widehat G$ by deleting the new vertex and these independent edges; thus
(d) holds. If $n$ is even, then $\widehat G$ has order $n$, and $G$ is
obtained from $\widehat G$ by deleting these independent edges; thus (e)
holds.
\end{proof}

\subsection{Proof of Lemma~\ref{lem:snark-reduction}}\label{SS:snark-proof}

In this subsection we prove Lemma~\ref{lem:snark-reduction}. For every
cyclic $3$-edge-cut of $G$, some added edge must have its ends on
different sides of the cut, or the new vertex must have neighbors on both
sides, and the main difficulty is that these cuts may cross each other.
We first show that they become pairwise noncrossing after certain
vertices of degree $2$ are smoothed, so that they can be arranged in a tree, and we then choose
the new edges by a pairing argument on this tree.

Let $Q$ be a graph. To \emph{smooth} a vertex $v$ of degree $2$ in $Q$
whose two neighbors $x$ and $y$ are nonadjacent means to delete $v$ and
add the edge $xy$. If $Q$ is a connected subcubic cyclically
$3$-edge-connected graph, two cyclic $3$-edge-cuts $\delta_Q(A)$ and
$\delta_Q(B)$ of $Q$ \emph{cross} if the four sets $A\cap B$,
$A\setminus B$, $B\setminus A$, and $V(Q)\setminus(A\cup B)$ are all
nonempty. By Lemma~\ref{lem:cyc-2-3-cut-property}, the shores of a cyclic
$3$-edge-cut are determined up to complementation, and replacing $A$ or
$B$ by its complement only permutes these four sets, so this depends only
on the two cuts. A family of cyclic $3$-edge-cuts is \emph{laminar} if no
two of its members cross.

The next lemma shows that, although the cyclic $3$-edge-cuts of a
cyclically $3$-edge-connected $3$-critical graph may cross, they become
laminar after smoothing certain degree-$2$ vertices. It also shows that
cyclic $3$-edge-cuts survive the smoothing.

\begin{lem}\label{lem:smoothing-laminar}
Let $G$ be a cyclically $3$-edge-connected $3$-critical graph. Starting
with $Q=G$, repeatedly smooth a degree-$2$ vertex of the current graph $Q$
that is an end of an edge belonging to a cyclic $3$-edge-cut of $Q$, until
no such vertex remains, and let $H$ be the final graph. Then the following
hold.
\begin{enumerate}[(1)]
    \item Each smoothing in the process is well defined, and every graph
    in the process, including $H$, is a simple, subcubic,
    $2$-edge-connected, cyclically $3$-edge-connected graph in which no two
    degree-$2$ vertices are adjacent; in particular, the cyclic
    $3$-edge-cuts referred to in the process are well defined. Moreover,
    no degree-$2$ vertex of $H$ is an end of an edge belonging to a
    cyclic $3$-edge-cut of $H$.
    \item If $Q$ is one of the graphs in the process, including $G$, and
    $\delta_Q(S)$ is a cyclic $3$-edge-cut of $Q$ for some
    $S\subseteq V(Q)$, then $\delta_H(S\cap V(H))$ is a cyclic
    $3$-edge-cut of $H$.
    \item The cyclic $3$-edge-cuts of $H$ are laminar.
\end{enumerate}
\end{lem}

\begin{proof}
The graph $G$ is subcubic, and by
Lemma~\ref{lem:basic-3-critical}(1), every vertex of $G$ has degree at
least $2$ and $G$ is $2$-edge-connected. By
Lemma~\ref{lem:basic-3-critical}(3), no two degree-$2$ vertices of $G$
are adjacent.

We first show that one smoothing step preserves these properties
together with simplicity and cyclic $3$-edge-connectivity. Let $Q$ be a
current graph with all of these properties, and suppose that a degree-$2$
vertex $v$ of $Q$ is an end of an edge belonging to a cyclic
$3$-edge-cut of $Q$. Let the neighbors of $v$ be $x$ and $y$. Since no
two degree-$2$ vertices of $Q$ are adjacent, both $x$ and $y$ have degree
$3$. By Lemma~\ref{lem:cyc-2-3-cut-property}, every cyclic
$3$-edge-cut of $Q$ is a matching. This further implies that $x$ and $y$
are nonadjacent in $Q$: otherwise, a cyclic $3$-edge-cut of $Q$
containing an edge at $v$, which is an edge-cut by
Lemma~\ref{lem:cyc-2-3-cut-property}, would meet the triangle $xvyx$ in
an even number of edges, hence in two edges, and these share an end.
Hence the smoothing is well defined and does not create parallel
edges.

Let $Q'$ be the graph obtained from $Q$ by smoothing $v$. Then $Q'$ is
connected, since the path $xvy$ of $Q$ is replaced by the edge $xy$.
Smoothing $v$ does not change the degree of any vertex of $Q'$, and the
new edge $xy$ joins two vertices of degree $3$. Thus $Q'$ is subcubic
with minimum degree at least $2$, and no two degree-$2$ vertices of $Q'$
are adjacent.

Let $F=\delta_{Q'}(A)$ be an edge-cut of $Q'$ with
$\emptyset\ne A\ne V(Q')$. If $xy\notin F$, then $x$ and $y$ lie on the
same side of $F$, and $F$ is also an edge-cut of $Q$, obtained by putting
$v$ on that side. If $xy\in F$, say $x\in A$, then
$(F\setminus\{xy\})\cup\{xv\}=\delta_Q(A)$ is an edge-cut of $Q$. In both
cases the edge-cut of $Q$ has the same size as $F$ and two nonempty
sides, and a side of $F$ that contains a cycle of $Q'$ corresponds to a
side that contains a cycle of $Q$, obtained by replacing the edge $xy$
with the path $xvy$ if necessary.  Since $Q$ is $2$-edge-connected and cyclically $3$-edge-connected, so is $Q'$; for cyclic
$3$-edge-connectivity it suffices to consider edge-cuts, since if two
components of $Q'-F'$ contain cycles for some $F'\subseteq E(Q')$ with
$|F'|\le 2$, then for one such component $D$ the edge-cut
$\delta_{Q'}(V(D))\subseteq F'$ has a cycle on each side.

Next let $\delta_Q(S)$ be a cyclic $3$-edge-cut of $Q$ for some
$S\subseteq V(Q)$. We show that
$\delta_{Q'}(S\setminus\{v\})$ is a cyclic $3$-edge-cut of $Q'$. Since
$\delta_Q(S)$ is a matching, it contains at most one of $xv$ and $vy$. If
it contains neither, then $x$, $v$, and $y$ lie on the same side, and
$\delta_{Q'}(S\setminus\{v\})=\delta_Q(S)$. If it contains exactly one,
say $xv$, then $v$ and $y$ lie on one side and $x$ on the other, and
$\delta_{Q'}(S\setminus\{v\})$ is obtained from $\delta_Q(S)$ by replacing
$xv$ with $xy$. In both cases the new cut has three edges. Each side of
$\delta_Q(S)$ contains a cycle of $Q$. If such a cycle passes through
$v$, then it contains the path $xvy$, and replacing this path by the edge
$xy$ gives a cycle of $Q'$ on the same side, because $xy\notin E(Q)$.
Hence $\delta_{Q'}(S\setminus\{v\})$ is a cyclic $3$-edge-cut of $Q'$.

The process terminates because each smoothing reduces the number of
vertices. By induction on the number of smoothing steps, every graph in
the process, and in particular $H$, is simple, subcubic,
$2$-edge-connected, and cyclically $3$-edge-connected, and has no two adjacent
degree-$2$ vertices. By the stopping rule, no degree-$2$ vertex of $H$ is an end of
an edge belonging to a cyclic $3$-edge-cut of $H$. This proves (1).
Applying the preceding paragraph to each smoothing step after $Q$ proves
(2), since the vertex set of the current graph shrinks by the smoothed
vertex at each step.

It remains to prove (3).

Suppose, for a contradiction, that $H$ has two crossing cyclic
$3$-edge-cuts. Let these cuts be $\delta_H(A)$ and $\delta_H(B)$. Put
\[
X=A\cap B,\qquad
Y=A\setminus B,\qquad
Z=B\setminus A,\qquad
W=V(H)\setminus(A\cup B).
\]
See Figure~\ref{fig:laminarity-Gstar} for an illustration of these four sets.

Since $A$ and $B$ cross, all four sets $X,Y,Z,W$ are nonempty. Write
\[
e_{XY}=e_H(X,Y),\quad e_{XZ}=e_H(X,Z),\quad e_{XW}=e_H(X,W),
\]
and
\[
e_{YZ}=e_H(Y,Z),\quad e_{YW}=e_H(Y,W),\quad e_{ZW}=e_H(Z,W).
\]

\begin{figure}[ht]
\centering
\begin{tikzpicture}[
    region/.style={
        draw,
        rounded corners=5pt,
        minimum width=3.05cm,
        minimum height=1.05cm,
        align=center,
        font=\large
    },
    edge/.style={thick},
    elab/.style={font=\small, fill=white, inner sep=1pt},
    setbox/.style={
        draw,
        dashed,
        very thick,
        rounded corners=6pt,
        inner sep=8pt
    }
]

\node[region, fill=blue!8]   (X) at (0, 2.0) {$X=A\cap B$};
\node[region, fill=green!8]  (Y) at (-3.7, 0) {$Y=A\setminus B$};
\node[region, fill=orange!10](Z) at (3.7, 0) {$Z=B\setminus A$};
\node[region, fill=purple!8] (W) at (0,-2.0) {$W=V(H)\setminus(A\cup B)$};

\begin{scope}[on background layer]
    \node[setbox, blue!75, fit=(X)(Y), label={[blue!75]above left:$A=X\cup Y$}] {};
    \node[setbox, orange!90!black, fit=(X)(Z), label={[orange!90!black]above right:$B=X\cup Z$}] {};
\end{scope}

\draw[edge] (X) -- node[elab, above left] {$e_{XY}$} (Y);
\draw[edge] (X) -- node[elab, above right] {$e_{XZ}$} (Z);
\draw[edge] (X) -- node[elab, right, pos=.58] {$e_{XW}$} (W);

\draw[edge] (Y) to[out=18,in=162]
    node[elab, above, pos=.3] {$e_{YZ}$} (Z);

\draw[edge] (Y) -- node[elab, below left] {$e_{YW}$} (W);
\draw[edge] (Z) -- node[elab, below right] {$e_{ZW}$} (W);

\end{tikzpicture}
\caption{The four sets determined by two crossing shores $A$ and $B$.}
\label{fig:laminarity-Gstar}
\end{figure}
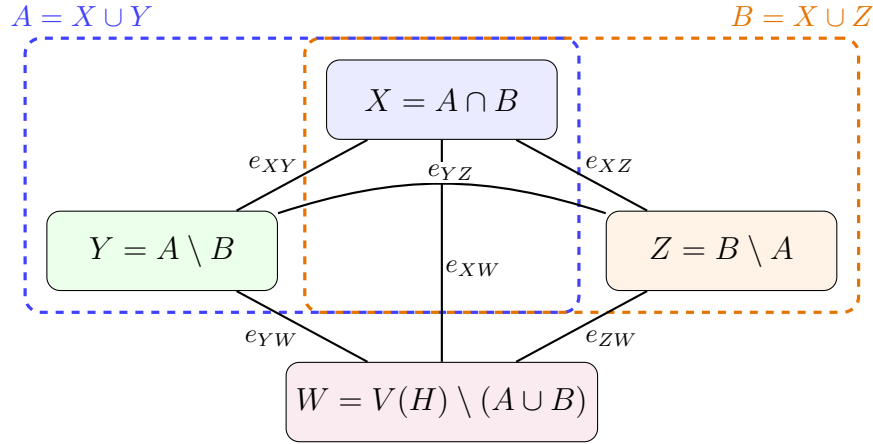

Since $\delta_H(A)$ and $\delta_H(B)$ both have size $3$, we have
\begin{equation}\label{eq:Gstar-cutAB}
e_{XZ}+e_{XW}+e_{YZ}+e_{YW}=3
\qquad\text{and}\qquad
e_{XY}+e_{XW}+e_{YZ}+e_{ZW}=3.
\end{equation}
The boundary sizes of the four corners are
\[
|\delta_H(X)|=e_{XY}+e_{XZ}+e_{XW},
\]
\[
|\delta_H(Y)|=e_{XY}+e_{YZ}+e_{YW},
\]
\[
|\delta_H(Z)|=e_{XZ}+e_{YZ}+e_{ZW},
\]
and
\[
|\delta_H(W)|=e_{XW}+e_{YW}+e_{ZW}.
\]
Adding these four equalities gives
\[
|\delta_H(X)|+|\delta_H(Y)|+|\delta_H(Z)|+|\delta_H(W)|
=
2(e_{XY}+e_{XZ}+e_{XW}+e_{YZ}+e_{YW}+e_{ZW}).
\]
On the other hand, from \eqref{eq:Gstar-cutAB},
\[
e_{XY}+e_{XZ}+e_{YW}+e_{ZW}+2e_{XW}+2e_{YZ}=6.
\]
Thus
\begin{equation}\label{eq:Gstar-corner-sum}
|\delta_H(X)|+|\delta_H(Y)|+|\delta_H(Z)|+|\delta_H(W)|
=
12-2e_{XW}-2e_{YZ}\le 12.
\end{equation}

We claim that some corner has boundary size at most $2$. Indeed, if all
four corners had boundary size at least $3$, then by
\eqref{eq:Gstar-corner-sum} all four boundary sizes would be exactly $3$,
and also
\[
e_{XW}=e_{YZ}=0.
\]
Then \eqref{eq:Gstar-cutAB} gives
\[
e_{XZ}+e_{YW}=3
\qquad\text{and}\qquad
e_{XY}+e_{ZW}=3.
\]
The four corner-boundary equalities become
\[
e_{XY}+e_{XZ}=3,\qquad
e_{XY}+e_{YW}=3,
\]
and
\[
e_{XZ}+e_{ZW}=3,\qquad
e_{YW}+e_{ZW}=3.
\]
From the first two equations, we get $e_{XZ}=e_{YW}$, which contradicts
$e_{XZ}+e_{YW}=3$. Hence some corner has boundary size at most $2$.

By replacing $A$ or $B$ by its complement, and by interchanging $A$ and
$B$ if necessary, we may assume that this corner is $Z$. Indeed, if the
small corner is $X$, replace $A$ by $V(H)\setminus A$; if it is $Y$,
interchange $A$ and $B$; and if it is $W$, replace $B$ by
$V(H)\setminus B$. These replacements do not change the two cuts, so
they remain crossing cyclic $3$-edge-cuts, and $H[A]$ and $H[B]$ still
contain cycles. Therefore we may assume that $|\delta_H(Z)|\le 2$.

Since $H$ is $2$-edge-connected, and since $Z$ is nonempty and
proper, we have $|\delta_H(Z)|=2$.
Since $H[A]$ contains a cycle and $A\subseteq V(H)\setminus Z$, if
$H[Z]$ contained a cycle, then $\delta_H(Z)$ would be a cyclic
disconnecting set of size $2$, contradicting the cyclic
$3$-edge-connectivity of $H$. Hence
$H[Z]$ is acyclic.

Also $H[Z]$ is connected. Otherwise one component of $H[Z]$ would have
boundary size at most $1$, contradicting that $H$ is
$2$-edge-connected. Hence $H[Z]$ is a tree. Since $|\delta_H(Z)|=2$, we have
\[
\sum_{v\in Z}d_H(v)
=
2e_H(Z)+|\delta_H(Z)|
=
2(|Z|-1)+2
=
2|Z|.
\]
As $H$ is $2$-edge-connected, every vertex of $H$ has degree at least $2$. Hence
every vertex of $Z$ has degree exactly $2$ in $H$.

If $|Z|\ge 2$, then the tree $H[Z]$ contains an edge whose two ends both
have degree $2$ in $H$, contradicting that no two degree-$2$ vertices of
$H$ are adjacent. Therefore $Z$ consists of a single vertex, say
$Z=\{z\}$, and $d_H(z)=2$.

Now the two edges incident with $z$ lie in the crossing cyclic
$3$-edge-cuts $\delta_H(A)$ and $\delta_H(B)$. Indeed, an edge from
$z$ to $X$ lies in $\delta_H(A)$, an edge from $z$ to $W$ lies in
$\delta_H(B)$, and an edge from $z$ to $Y$ lies in both $\delta_H(A)$ and
$\delta_H(B)$. Thus $z$ is a degree-$2$ end of an edge belonging to
a cyclic $3$-edge-cut of $H$. This contradicts the choice of the final
graph $H$, in which no such degree-$2$ vertex remains.

Therefore no two cyclic $3$-edge-cuts of $H$ cross. Hence the cyclic
$3$-edge-cuts of $H$ are laminar.
\end{proof}

\begin{proof}[Proof of Lemma~\ref{lem:snark-reduction}]
By Lemma~\ref{lem:basic-3-critical}(1), every vertex of $G$ has degree
$2$ or $3$, and $G$ is of class $2$ because it is $3$-critical. Let $t$ be
the number of degree-$2$ vertices of $G$. Since
$2|E(G)|=3|V(G)|-t$, the numbers $t$ and $|V(G)|$ have the same parity,
and $t\ge 3$ by Lemma~\ref{lem:basic-3-critical}(5).

Every graph $\widehat G$ obtained as in (1) or (2) is cubic. We call the
added edges, including in case~(2) the three edges at the new vertex $x$,
the \emph{completion edges}. We choose the pairing, and in case~(2) the
vertices $z_1,z_2,z_3$, so that the following property holds.
\begin{enumerate}
    \item[(P)] For every cyclic $3$-edge-cut $\delta_G(X)$ of $G$, some
    completion edge joining two vertices of $G$ has exactly one end in
    $X$, or the new vertex $x$ exists and has a neighbor in $X$ and a
    neighbor in $V(G)\setminus X$.
\end{enumerate}
If $G$ has no cyclic $3$-edge-cut, then (P) holds for every choice, and we
choose the pairing and $z_1,z_2,z_3$ arbitrarily. Otherwise we construct a
choice satisfying (P) in Steps~1--5 below. In Steps~6 and~7 we show that
every $\widehat G$ satisfying (P) is a snark. Until the end of Step~5 we
assume that $G$ has a cyclic $3$-edge-cut.

\medskip
\noindent\textbf{Step 1: Both shores of every cyclic $3$-edge-cut of the
smoothed graph contain a degree-$2$ vertex.}
Let $G^s$ be a graph $H$ obtained from $G$ by the smoothing process of
Lemma~\ref{lem:smoothing-laminar}; the argument below works for every run
of the process. Let $S$ be the set of degree-$2$
vertices of $G^s$, and let $R=V(G)\setminus V(G^s)$ be the set of
smoothed vertices. Smoothing deletes only degree-$2$ vertices and does not
change the degree of any remaining vertex, so the degree-$2$ vertices of
$G$ are exactly the vertices of $S\cup R$. The neighbors in $G$ of a
vertex $r\in R$ have degree $3$, so they are never smoothed, and the two
edges of $G$ at $r$ remain in the current graph until $r$ is smoothed.
Hence smoothing $r$ creates the edge $ab$, where $a$ and $b$ are the
neighbors of $r$ in $G$, and this edge belongs to $G^s$. Consequently,
since smoothing removes only the two edges at the smoothed vertex, the
edges of $G^s$ are exactly the edges of $G$ joining two vertices of
$G^s$ together with the edges created by the vertices of $R$, each
created by a unique vertex of $R$.

We claim that each shore of every cyclic $3$-edge-cut of $G^s$ contains a
vertex of $S$. Suppose not, and choose a shore $X$ of a cyclic
$3$-edge-cut of $G^s$ with $X\cap S=\emptyset$ and with $|X|$ minimum.
Let $X^+$ be the set obtained from $X$ by adding every vertex of $R$ whose
two neighbors in $G$ both lie in $X$. An edge of $G$ between two vertices
of $G^s$ lies in $\delta_G(X^+)$ if and only if it lies in
$\delta_{G^s}(X)$. For $r\in R$ with neighbors $a$ and $b$ in $G$, exactly
one of $ra$ and $rb$ lies in $\delta_G(X^+)$ if exactly one of $a,b$ lies
in $X$, that is, if the created edge $ab$ lies in $\delta_{G^s}(X)$, and
neither lies in $\delta_G(X^+)$ otherwise. Hence
$|\delta_G(X^+)|=|\delta_{G^s}(X)|=3$. Moreover, replacing each created
edge $ab$ of a cycle of $G^s$ by the path $arb$ of $G$ gives a cycle of
$G$ on the same side, so both sides of $\delta_G(X^+)$ contain cycles.
Thus $\delta_G(X^+)$ is a cyclic $3$-edge-cut of $G$, and by the
hypothesis of the lemma, $X^+$ contains a vertex of degree $2$ in $G$.
Since $X\cap S=\emptyset$, this vertex is some $r\in R$ whose two
neighbors $a$ and $b$ in $G$ both lie in $X$.

When $r$ was smoothed, one of the edges $ra$ and $rb$ belonged to a cyclic
$3$-edge-cut $\delta_Q(X_r)$ of the current graph $Q$. By
Lemma~\ref{lem:cyc-2-3-cut-property}, this cut is a matching, so it
contains exactly one of $ra$ and $rb$, and $a$ and $b$ lie on different
sides of it. By Lemma~\ref{lem:smoothing-laminar}(2), with
$X_r'=X_r\cap V(G^s)$, the cut $\delta_{G^s}(X_r')$ is a cyclic $3$-edge-cut of
$G^s$, and $a$ and $b$ lie on different sides of it. Replacing $X_r'$ by
its complement if necessary, we may assume that $a\in X_r'$ and
$b\notin X_r'$. By Lemma~\ref{lem:smoothing-laminar}(3), the cuts
$\delta_{G^s}(X)$ and $\delta_{G^s}(X_r')$ do not cross. Since
$a\in X\cap X_r'$ and $b\in X\setminus X_r'$, either $X_r'\setminus X=\emptyset$
or $V(G^s)\setminus(X\cup X_r')=\emptyset$. In the first case $X_r'$ is a
subset of $X$ not containing $b$, and in the second case
$V(G^s)\setminus X_r'$ is a subset of $X$ not containing $a$. In either case
$X$ properly contains a shore of a cyclic $3$-edge-cut of $G^s$, and this
shore contains no vertex of $S$. This contradicts the minimality of
$|X|$ and proves the claim.

\medskip
\noindent\textbf{Step 2: A tree representing the cyclic $3$-edge-cuts of
$G^s$.}
Since $G$ has a cyclic $3$-edge-cut, so does $G^s$ by
Lemma~\ref{lem:smoothing-laminar}(2), and hence $|S|\ge 2$ by
Step~1, as both shores of this cut contain a vertex of $S$. Fix $z_0\in S$, and for each cyclic $3$-edge-cut of $G^s$ choose
its shore not containing $z_0$. Let $\mathcal L$ be the family of chosen
shores. Any two members $X_1,X_2$ of $\mathcal L$ are nested or disjoint:
otherwise $X_1\cap X_2$, $X_1\setminus X_2$, and $X_2\setminus X_1$ would
be nonempty, and $z_0\notin X_1\cup X_2$, so the two cuts would cross,
contrary to Lemma~\ref{lem:smoothing-laminar}(3). Hence any two members
of $\mathcal L\cup\{V(G^s)\}$ with a common vertex are nested, so the
members containing a given vertex form a chain. Let $T$ be the rooted
tree with vertex set $\mathcal L\cup\{V(G^s)\}$ and root $V(G^s)$, in
which the parent of $B\in\mathcal L$ is the smallest member of
$\mathcal L\cup\{V(G^s)\}$ properly containing $B$. Let $e_B$ denote the
edge of $T$ joining $B$ to its parent. Then $B\mapsto e_B$ gives a
bijection from $\mathcal L$ to $E(T)$, and hence a bijection between
$E(T)$ and the set of cyclic $3$-edge-cuts of $G^s$, under which $e_B$
corresponds to $\delta_{G^s}(B)$. For $v\in V(G^s)$, let $L(v)$ be the
smallest member of $\mathcal L\cup\{V(G^s)\}$ containing $v$. Then, for
every $B\in\mathcal L$ and $v\in V(G^s)$,
\begin{equation}\label{eq:tree-sides}
v\in B \quad\text{if and only if}\quad L(v)\text{ lies in the component of $T-e_B$ containing $B$.}
\end{equation}
Indeed, the component of $T-e_B$ containing $B$ consists of the members
of $\mathcal L$ contained in $B$. If $v\in B$, then $L(v)\subseteq B$
by the minimality of $L(v)$, and conversely, if $L(v)\subseteq B$,
then $v\in B$.

We place a \emph{mark} at the vertex $L(z)$ of $T$ for each $z\in S$,
and we identify the mark with $z$. By Step~1 and~\eqref{eq:tree-sides},
every edge of $T$ \emph{separates marks}, that is, both components of
$T-e$ contain a mark for every $e\in E(T)$.

\medskip
\noindent\textbf{Step 3: A pairing fact for marked trees.}
We use the following elementary fact. Let $T_0$ be a finite tree carrying
at least two marks, each placed at a vertex of $T_0$, where several marks
may be placed at the same vertex, and suppose that every edge of $T_0$
separates marks. If the number of marks is even, then the marks can be
partitioned into pairs so that every edge $e$ of $T_0$ \emph{separates} at
least one pair, that is, the two marks of the pair lie in different
components of $T_0-e$. If the number of marks is odd, then three marks,
called \emph{special}, can be chosen and the remaining marks partitioned
into pairs so that every edge $e$ of $T_0$ either separates one of the
pairs or separates the special marks \emph{nontrivially}, that is, each
component of $T_0-e$ contains a special mark.

We prove the fact by induction on the number of marks. First we reduce to
the case where every mark lies at a leaf and every leaf carries exactly
one mark. Indeed, attach a new leaf for each mark and move the mark to
that leaf; every edge of the enlarged tree still separates marks, since
there are at least two marks. If the desired pairing, or the desired choice of special marks
together with a pairing of the remaining marks, works in the enlarged
tree, then it also works for the original tree, whose edges are edges of
the enlarged tree. Thus we may assume that every mark lies at a leaf and
that distinct marks lie at distinct leaves. Moreover, every leaf of $T_0$
carries a mark, since an unmarked leaf would give an incident edge one of
whose sides contains no mark.

If $T_0$ has two marks, then $T_0$ is a path whose ends carry the marks, and
we pair them. If $T_0$ has three marks, then we choose all three as the
special marks; every edge then separates them nontrivially. Thus assume
that $T_0$ has at least four marks.

Let $\ell$ be a leaf of $T_0$, and let $m_\ell$ be the mark at $\ell$.
Delete $\ell$, and then repeatedly delete any unmarked leaf that is
created. Let $T_0'$ be the resulting tree. Then $T_0'$ carries all marks
except $m_\ell$, all at leaves of $T_0'$, and every leaf of $T_0'$ carries a
mark. Hence every edge of $T_0'$ separates the remaining marks, since each
component of $T_0'-e$ contains a leaf of $T_0'$. Therefore the induction
hypothesis applies to $T_0'$. Let $q$ be the vertex of $T_0'$ at which the
deleted path from $\ell$ meets $T_0'$.

First suppose that the number of marks of $T_0$ is odd. Then the number of
marks of $T_0'$ is even. By induction, the marks of $T_0'$ can be paired so
that every edge of $T_0'$ separates at least one pair. Choose one pair
$\{a,b\}$ from this pairing. In $T_0$, choose $m_\ell,a,b$ as the special
marks, and keep all pairs of the inductive pairing except $\{a,b\}$.
Every edge on the deleted path from $\ell$ to $q$ separates $m_\ell$ from
$a$ and $b$, and hence separates the special marks nontrivially. Now let
$e\in E(T_0')$. If $e$ separates one of the kept pairs, there is nothing to
show. Otherwise, in the inductive pairing, $e$ separates the removed pair
$\{a,b\}$. Hence $a$ and $b$ lie in different components of $T_0-e$, and
the special marks $m_\ell,a,b$ are separated nontrivially by $e$. This
proves the odd case.

Now suppose that the number of marks of $T_0$ is even. Then the number of
marks of $T_0'$ is odd. By induction, there are special marks $a,b,c$ of
$T_0'$, and the remaining marks of $T_0'$ can be paired so that every edge of
$T_0'$ either separates one of the pairs or separates $a,b,c$
nontrivially. Let $U$ be the minimal subtree of $T_0'$ containing the
leaves carrying $a$, $b$, and $c$, and let $u$ be the unique vertex of
degree $3$ in $U$. Then $U$ is the union of the three paths from $u$ to
these leaves, and these paths lie, apart from $u$, in three distinct
components of $T_0'-u$. At most one of these components contains $q$, so
we may relabel $a,b,c$ so that the component of $T_0'-u$ containing $a$
does not contain $q$.

Pair $m_\ell$ with $a$, pair $b$ with $c$, and keep all other pairs given
by the induction on $T_0'$. Every edge on the deleted path from $\ell$ to
$q$ separates the pair $\{m_\ell,a\}$. Now let $e\in E(T_0')$. If $e$
separates one of the inductive pairs, there is nothing to show.
Otherwise, $e$ separates $a,b,c$ nontrivially, so $e$ lies in $U$, on
exactly one of the paths from $u$ to $a$, $b$, and $c$. If $e$ lies on
the path from $u$ to $b$ or on the path from $u$ to $c$, then $e$
separates the pair $\{b,c\}$. Suppose that $e$ lies on the path from $u$
to $a$, and let $A_e$ and $B_e$ be the two components of $T_0'-e$, with
$a\in A_e$ and $u,b,c\in B_e$. The component $A_e$ is contained in the
component of $T_0'-u$ containing $a$, which does not contain $q$. Hence
$q\in B_e$. Since the deleted path from $\ell$ to $q$ meets $T_0'$ only at
$q$, the mark $m_\ell$ lies in the component of $T_0-e$ containing $B_e$.
Thus $e$ separates the pair $\{m_\ell,a\}$. This proves the even case,
and hence the pairing fact.

\medskip
\noindent\textbf{Step 4: Choice of the completion.}
Apply the pairing fact of Step~3 to the tree $T$ of Step~2 with the marks
$S$. This is possible since $|S|\ge 2$ and every edge of $T$ separates
marks, by Step~2. If $|R|$ is even, pair
the vertices of $R$ arbitrarily. If $|R|$ is odd, choose $w\in R$, let
$a_w$ and $b_w$ be the neighbors of $w$ in $G$, and pair the vertices of
$R\setminus\{w\}$ arbitrarily. Recall that $t=|S|+|R|$ has the same
parity as $|V(G)|$. We complete the choice as follows.
\begin{enumerate}[(i)]
    \item If $|S|$ and $|R|$ are even, then $|V(G)|$ is even, and we take
    the pairs of $S$ given by the even case of the pairing fact.
    \item If $|S|$ is odd and $|R|$ is even, then $|V(G)|$ is odd. We
    take the special marks $z_1,z_2,z_3$ and the pairs of the remaining
    marks given by the odd case of the pairing fact.
    \item If $|S|$ is even and $|R|$ is odd, then $|V(G)|$ is odd. We take
    the pairs of $S$ given by the even case of the pairing fact, choose
    one of them, say $\{p_1,p_2\}$, and let $z_1=p_1$, $z_2=p_2$, and $z_3=w$;
    the other pairs of $S$ are kept.
    \item If $|S|$ and $|R|$ are odd, then $|V(G)|$ is even. Let
    $s_1,s_2,s_3$ be the special marks and take the pairs of the remaining
    marks given by the odd case of the pairing fact. Let $u$ be the
    vertex of $T$ that lies on each of the three paths of $T$ joining two
    of $L(s_1),L(s_2),L(s_3)$, and for $j\in[3]$ let $P_j$ be the
    path of $T$ from $u$ to $L(s_j)$. These three paths are pairwise
    edge-disjoint, and those of positive length lie, apart from $u$, in
    distinct components of $T-u$. The vertices $L(a_w)$ and
    $L(b_w)$ lie in at most two components of $T-u$. Hence there is
    $i_0\in[3]$ such that either $P_{i_0}$ has length $0$, or the component of
    $T-u$ containing $P_{i_0}-u$ contains neither $L(a_w)$ nor
    $L(b_w)$. We pair $w$ with $s_{i_0}$ and pair the other two special
    marks with each other.
\end{enumerate}
In cases (i) and (iv), every degree-$2$ vertex of $G$ lies in exactly one
pair, and $\widehat G$ is obtained from $G$ by adding an edge for each
pair, as in~(1). In cases (ii) and (iii), every degree-$2$ vertex of $G$
other than $z_1,z_2,z_3$ lies in exactly one pair, and $\widehat G$ is
obtained from $G$ by adding a new vertex $x$ adjacent to $z_1,z_2,z_3$
and an edge for each pair, as in~(2).

\medskip
\noindent\textbf{Step 5: The completion satisfies (P).}
Let $\delta_G(X)$ be a cyclic $3$-edge-cut of $G$. By
Lemma~\ref{lem:smoothing-laminar}(2), $\delta_{G^s}(X\cap V(G^s))$ is a
cyclic $3$-edge-cut of $G^s$; let $e$ be the corresponding edge of $T$.
By~\eqref{eq:tree-sides}, two vertices $v,v'\in V(G^s)$ lie on different
sides of $\delta_G(X)$ if and only if $L(v)$ and $L(v')$ lie in
different components of $T-e$. In particular, if $e$ separates two marks
$z,z'\in S$, then exactly one of $z$ and $z'$ lies in $X$.

In case~(i), $e$ separates some pair $\{z,z'\}$, and the completion edge
$zz'$ has exactly one end in $X$. In case~(ii), either $e$ separates some
pair, which gives a completion edge as before, or $e$ separates
$z_1,z_2,z_3$ nontrivially, and then $x$ has a neighbor in $X$ and a
neighbor in $V(G)\setminus X$. In case~(iii), $e$ separates some pair of
the even pairing of $S$. If this pair is not $\{p_1,p_2\}$, it gives a
completion edge with exactly one end in $X$; if it is $\{p_1,p_2\}$, then
$x$ is adjacent to $p_1$ and $p_2$, which lie on different sides of
$\delta_G(X)$.

Finally consider case~(iv). If $e$ separates one of the pairs of
non-special marks, we are done as before. Otherwise $e$ separates
$s_1,s_2,s_3$ nontrivially, so $e$ lies on a path of $T$ joining two of
$L(s_1),L(s_2),L(s_3)$, and hence on $P_{j_0}$ for exactly one
$j_0\in[3]$. Let $D_e$ be the component of $T-e$ not containing $u$. Then
$L(s_{j_0})\in D_e$, while $L(s_k)\notin D_e$ for $k\ne j_0$, because the
path $P_k$ from $u$ to $L(s_k)$ does not use $e$. If $j_0\ne i_0$, then
$e$ separates $s_{j_0}$ from the special mark $s_k$, where
$\{i_0,j_0,k\}=[3]$, and the completion edge $s_{j_0}s_k$ has exactly
one end in $X$.
If $j_0=i_0$, then $P_{i_0}$ has positive length, and $D_e$ is contained in the
component of $T-u$ containing $P_{i_0}-u$. By the choice of $i_0$, neither
$L(a_w)$ nor $L(b_w)$ lies in $D_e$. Hence $a_w$ and $b_w$ lie on
the same side of $\delta_G(X)$, and $s_{i_0}$ lies on the other side. Since
$\delta_G(X)$ is a matching by Lemma~\ref{lem:cyc-2-3-cut-property}, it
does not contain both $wa_w$ and $wb_w$, so $w$ lies on the side of
$a_w$ and $b_w$. Thus the completion edge $ws_{i_0}$ has exactly one end in
$X$. This proves that the choice satisfies (P).

\medskip
\noindent\textbf{Step 6: $\widehat G$ is simple, has girth at least $4$,
and is of class $2$.}
By Lemma~\ref{lem:basic-3-critical}(3), no two degree-$2$ vertices of $G$
are adjacent or have a common neighbor. Each degree-$2$ vertex of $G$
is an end of exactly one completion edge. Therefore no completion edge
duplicates an edge of $G$, and no triangle of $\widehat G$ contains a
completion edge $pq$ with $p,q\in V(G)$: its third vertex would be a
common neighbor of $p$ and $q$ in $G$, because $p$ and $q$ are incident
with no other completion edge. In the case where the new vertex $x$ is added, its three
neighbors are degree-$2$ vertices of $G$, and these are pairwise
nonadjacent in $G$ and, since each of them is an end of exactly one
completion edge, also in $\widehat G$. Thus no triangle passes
through $x$. Since $G$ has girth at
least $4$, it follows that $\widehat G$ has girth at least $4$. Finally,
$\widehat G$ is of class $2$ because it is cubic and contains the
class~$2$ graph $G$.

\medskip
\noindent\textbf{Step 7: $\widehat G$ is cyclically $4$-edge-connected.}
Suppose, for a contradiction, that $\widehat G$ has an edge set $F_0$
with $|F_0|\le 3$ such that two components $D$ and $D'$ of
$\widehat G-F_0$ contain cycles. Put $A=V(D)$ and
$F=\delta_{\widehat G}(A)$. Then $F\subseteq F_0$, so $|F|\le 3$, and
both $\widehat G[A]\supseteq D$ and $\widehat G-A\supseteq D'$ contain a cycle;
note that $\widehat G[A]$ and $\widehat G-A$ are the two sides of
$\widehat G-F$. Put
\[
A_0=A\cap V(G)
\qquad\text{and}\qquad
C_0=\delta_G(A_0).
\]
Then $C_0\subseteq F$, because every edge of $G$ crossing between $A_0$
and $V(G)\setminus A_0$ also crosses between $A$ and
$V(\widehat G)\setminus A$. Hence $|C_0|\le 3$.

We claim that both $G[A_0]$ and $G[V(G)\setminus A_0]$ contain a cycle.
Suppose first that $G[A_0]$ is acyclic. Since $|\delta_G(A_0)|=|C_0|\le 3$,
Lemma~\ref{lem:basic-3-critical}(6) shows that $A_0$ contains at most one
degree-$2$ vertex of $G$. The $A$-side $\widehat G[A]$ of
$\widehat G-F$ contains a cycle. This cycle cannot lie
entirely in $G[A_0]$, so it uses a completion edge. Every edge of this
cycle has both ends in $A$, so a completion edge of the cycle joining
two vertices of $G$ has both ends in $A_0$, giving at least two
degree-$2$ vertices of $G$ in $A_0$. In the case where the cycle uses
the new vertex $x$, it uses two edges incident with $x$, and their ends
other than $x$ again give at least two degree-$2$ vertices of $G$ in
$A_0$. Both possibilities contradict
Lemma~\ref{lem:basic-3-critical}(6). Therefore $G[A_0]$ contains a cycle. The same
argument applied to the other side shows that $G[V(G)\setminus A_0]$
contains a cycle.

Thus $C_0=\delta_G(A_0)$ is a cyclic disconnecting set of $G$ with
$|C_0|\le 3$. Since $G$ is cyclically $3$-edge-connected, we have
$|C_0|=3$. Hence $C_0$ is a cyclic $3$-edge-cut of $G$, and by~(P) some
completion edge lies in $F$.
Indeed, either a completion edge joining two vertices of $G$ has exactly
one end in $A_0$, and then it lies in $F$; or $x$ has a neighbor in $A_0$
and a neighbor in $V(G)\setminus A_0$, and then, whichever side of $F$
contains $x$, one of the edges at $x$ lies in $F$.

But $C_0\subseteq F$ and the three edges of $C_0$ are edges of $G$, while
the completion edge is not an edge of $G$. Hence
$|F|\ge |C_0|+1=4$,
contradicting $|F|\le 3$. Therefore, for every edge set $F_0$ of
$\widehat G$ with $|F_0|\le 3$, at most one component of
$\widehat G-F_0$ contains a cycle, and so $\widehat G$ is cyclically
$4$-edge-connected.

Combining Steps~6 and~7, $\widehat G$ is a cyclically
$4$-edge-connected cubic class~$2$ graph of girth at least $4$, that is,
a snark.
\end{proof}

\section*{Declaration of Use of AI Tools}

All mathematical ideas in this paper belong to the authors. AI tools,
including ChatGPT~5.5 and Claude, were used to generate all nontrivial
$3$-critical graphs of orders $13$ to $24$; in particular, they were used
to develop, debug, and document the search, classification, and audit
code, which is archived in the data repository (see the data-and-code
availability statement in Appendix~\ref{S:algorithmic}). These graphs
served as a database that the authors analyzed in obtaining the
characterization in Theorem~\ref{thm:characterization}, in testing
Conjecture~\ref{conj:petersen-minor}, and in developing the compatibility
conditions on edge colorings used in
Theorem~\ref{thm:compatible-extension}. Claude was also used to proofread
and revise the manuscript. The authors reviewed, tested, and verified all
AI-assisted code, and independently checked all definitions, statements,
proofs, and computational results. The authors assume responsibility for all
content.

\section*{Acknowledgments}
We thank Dr.~Eckhard Steffen for bringing the
reference~\cite{FioriniWilson1977} to our attention, in particular for
pointing out that the Haj\'os-join decomposition in
Theorem~\ref{lem:Hajos2} is Theorem~12.4 there. This work was
completed in part with resources provided by the Auburn University
Easley Cluster.

\bibliographystyle{abbrvurl}
\bibliography{3-critical}

\appendix

\section{Algorithmic aspects of the enumeration}\label{S:algorithmic}

In this appendix we describe the computer search behind
Table~\ref{table:enumeration}; the search enumerates the nontrivial
$3$-critical graphs of orders $13$ through $24$. Subsection~\ref{SS:pipeline}
describes the search pipeline, Subsection~\ref{SS:per-order-stats} records
the number of graphs passing each stage, and Subsection~\ref{S:census-map}
classifies the graphs found according to the first applicable case of the
proof of Theorem~\ref{thm:characterization}.

\textbf{Data and code availability.}  All census output files,
machine-readable audit reports, and search and audit code referenced in this paper
are archived in the public repository
\begin{center}
  \url{https://github.com/chenle02/edge-3-critical-graphs-data}
\end{center}
(release \texttt{v1.3.0}).  This release is archived on Zenodo with
version DOI \texttt{10.5281/zenodo.23678968}~\cite{edge3critical-data};
the concept DOI \texttt{10.5281/zenodo.20821990} resolves to the latest
archived version.  Throughout this appendix, file paths of the forms
\texttt{results/\dots}, \texttt{reports/\dots}, and \texttt{code/\dots} are
relative to this repository, which we call the \emph{data repository}.  The
README records SHA-256 hashes of the census files of orders $13$, $15$,
$17$, and $19$--$24$ and gives a file-to-claim mapping.  The implementation is under \texttt{code/}.  Its input and output are
graph6 strings. The census counts in Tables~\ref{table:enumeration},
\ref{tab:per-order-stats}, and~\ref{tab:per-order-stats-even} are taken
from the per-order JSON files under \texttt{results/}, and the other
counts from the reports named where they are used.

\subsection{Pipeline overview}\label{SS:pipeline}

The pipeline processes each target order $n$ in five stages.

\begin{enumerate}
  \item \textbf{Generate.}  The program calls \texttt{geng} from the
  \texttt{nauty} package~\cite{McKayPiperno2014} with the flags
  \[
    \texttt{-Cq -d2 -D3 } n .
  \]
  Thus only $2$-connected graphs with minimum degree at least $2$ and maximum
  degree at most $3$ are generated.  The generation command is assembled in
  \texttt{code/main.py}.  For example, at order $13$ this generation
  step produces $11{,}679$ candidates.

  \item \textbf{Prune.}  The cycle $C_n$, the only generated graph of
  maximum degree $2$, is skipped, since the search requires maximum degree
  exactly $3$; it is counted in the columns ``Generated'' and ``After
  pruning'' of the tables below, but in none of the later columns.
  Every other generated graph is discarded if it satisfies either of the
  following two filters from
  \[
    \texttt{code/critical\_graph\_search/pruning.py}.
  \]
  \begin{enumerate}
    \item[(F1)] It is bipartite.  By K\"onig's Theorem, every
    bipartite graph is class~$1$~\cite{MR1511872}.

    \item[(F2)] It is regular, that is, cubic.  Cubic graphs are not
    $3$-critical by Lemma~\ref{lem:basic-3-critical}(2).
  \end{enumerate}

  \item \textbf{Class~$2$ test.}  For each remaining candidate, the program
  checks whether $G$ is $\Delta(G)$-edge-colorable by backtracking.  If such a
  coloring exists, then $\chi'(G)=\Delta(G)$ and the graph is discarded.

  \item \textbf{Edge-criticality test.}  For each class~$2$ candidate, the
  program verifies that $G-e$ is $\Delta(G)$-edge-colorable for every edge
  $e\in E(G)$.  This is the implemented $\Delta$-criticality test in
  \[
    \texttt{code/critical\_graph\_search/criticality.py}.
  \]
  The primary census uses a bitmask backtracking implementation.  It and
  a simpler reference backtracking implementation are both in
  \[
    \texttt{code/critical\_graph\_search/edge\_coloring.py},
  \]
  and the reference implementation is used as a cross-check by two of the
  audit scripts under \texttt{code/scripts/}.

  \item \textbf{Overfull-subgraph test.}  Finally, the program searches all
  induced subgraphs $H$ of odd order and discards $G$ if
  \[
    |E(H)|>\Delta(G)\left\lfloor\frac{|V(H)|}{2}\right\rfloor.
  \]
  Such a graph contains a $\Delta(G)$-overfull subgraph and is therefore
  trivial.  The remaining graphs, which we call \emph{survivors}, are the
  nontrivial $3$-critical graphs of order $n$.
\end{enumerate}

The per-order statistics produced by these stages are recorded in
Subsection~\ref{SS:per-order-stats}.

\subsection{Per-order statistics}\label{SS:per-order-stats}

Table~\ref{tab:per-order-stats} records the quantitative throughput of the
pipeline for the six odd orders used in the census.  The column ``After
pruning'' is the generated count minus the graphs removed by the two pruning
filters~(F1) and~(F2); it still includes the cycle $C_n$.  The column ``Class 1 / non-critical'' records the
remaining graphs that are either $3$-edge-colorable or fail the
edge-criticality test.

\begin{table}[h]
  \centering
  \scriptsize
  \setlength{\tabcolsep}{3pt}
  \caption{Per-order pipeline statistics for $\Delta=3$ at odd orders.  Source
  data are in the data repository directory \texttt{results/}.}
  \label{tab:per-order-stats}
  \vspace{4pt}
  \begin{tabular}{c|r|r|r|r|r|r}
    \hline
    Order $n$ & Generated & After pruning & Class 1 / non-critical &
    Total critical & Overfull critical & Nontrivial survivors \\
    \hline
    $13$ & $11{,}679$ & $11{,}380$ & $10{,}624$ & $755$ & $741$ & $14$ \\
    $15$ & $165{,}993$ & $163{,}531$ & $156{,}559$ & $6{,}971$ & $6{,}877$ & $94$ \\
    $17$ & $2{,}756{,}486$ & $2{,}732{,}845$ & $2{,}655{,}174$ & $77{,}670$ & $76{,}896$ & $774$ \\
    $19$ & $51{,}643{,}246$ & $51{,}387{,}795$ & $50{,}380{,}366$ & $1{,}007{,}427$ & $1{,}000{,}443$ & $6{,}984$ \\
    $21$ & $1{,}068{,}435{,}908$ & $1{,}065{,}397{,}521$ & $1{,}050{,}610{,}728$ & $14{,}786{,}792$ & $14{,}716{,}262$ & $70{,}530$ \\
    $23$ & $24{,}086{,}917{,}749$ & $24{,}047{,}687{,}763$ & $23{,}806{,}795{,}495$ & $240{,}892{,}267$ & $240{,}110{,}081$ & $782{,}186$ \\
    \hline
  \end{tabular}
\end{table}

The number of generated candidates grows from $11{,}679$ at order $13$ to
$24{,}086{,}917{,}749$ at order $23$.  In every row of
Table~\ref{tab:per-order-stats}, more than $98\%$ of the critical graphs are
overfull, so most graphs passing the class~$2$ and edge-criticality tests
are removed by the overfull-subgraph test.  From one odd order to the next,
the number of generated candidates grows by a factor between about $14$ and
about $23$, and the number of survivors by a factor between about $7$ and
about $11$.  The order-$21$ and order-$23$ censuses used about
$9.0\times 10^{3}$ and $7.2\times 10^{5}$ allocated core-hours,
respectively, as recorded in
\path{reports/order21_computational_effort_20260722.json} and
\path{reports/order23_computational_effort_20260723.json}; the
order-$23$ census was split into $64$ independent
slices.  These growth rates make the exhaustive search with this
pipeline expensive beyond the orders considered here; the order-$24$
census described below used a much faster native implementation.

For reproducibility, the SHA-256 hashes of the source JSON files
\texttt{order\_}$n$\texttt{\_delta\_3.json} or compressed
\texttt{order\_}$n$\texttt{\_delta\_3.json.gz}, for $n=13,15,17,19,21,23$, are respectively
\[
\begin{gathered}
\texttt{799aae0712bfb53b10279cdb178abd9ae2b55924f134d2fd3f9154ad4527ef10},\\
\texttt{c5c391d32a4019a6765e236ac9bfbd292572722073fbb72312d4ecbf91293162},\\
\texttt{c7447e9626f53cdb8a381376ec76eb95a58368bab85dd16dd02ba6d4f7b9a269},\\
\texttt{9f4eff7e13636fce3bcd3fc69cd2a3dfb36f1ad80e01656509fc6b9927f92b1e},\\
\texttt{2f9c3e46d744dc0b62a95631e050af657806ec76ae03a286f7e845c69cff24db},\\
\texttt{d5c9cf97ab8c2b693a24ecd4970e328a51210b302774b5506a200f7371a28896}.
\end{gathered}
\]
In each file, the critical graphs split exactly as
\[
  \text{Total critical}
  =
  \text{Overfull critical}
  +
  \text{Nontrivial survivors}.
\]
Moreover, in every row the entry ``After pruning'' exceeds the sum of
the entries ``Class 1 / non-critical'' and ``Total critical'' by $1$,
except at $n=19$, where it exceeds it by $2$. In each row one unit is the
cycle $C_n$, which is skipped before the class~$2$ test. The second unit
at $n=19$ is a bookkeeping artifact of an interrupted and resumed run:
the archived entry ``Class 1 / non-critical'' is one less than the true
value $50{,}380{,}367$ (the entry ``After pruning'' is exact, since the
$255{,}451$ pruned graphs are precisely the bipartite graphs generated by
\texttt{geng} with the flags above and \texttt{-b}). A fresh replay of order $19$ with the native implementation described
below reproduced the same $6{,}984$ survivors
(\texttt{reports/order19\_native\_replay\_20260728.md}).

We also ran the same pipeline at the even orders from $14$ to $22$, and
the native implementation described below at order $24$.  By
the result of Brinkmann and Steffen~\cite{BrinkmannSteffen1997}, there is
no $3$-critical graph of even order below $22$, exactly one of order
$22$, and exactly nine of order $24$, all nontrivial; our pipeline
reproduces this.  Table~\ref{tab:per-order-stats-even} records the
even-order throughput.

\begin{table}[h]
  \centering
  \scriptsize
  \setlength{\tabcolsep}{3pt}
  \caption{Per-order pipeline statistics for $\Delta=3$ at even orders.  Source
  data are in the data repository directory \texttt{results/}.  At order $24$
  the native streaming filter retained only the survivors (see the text
  below), so the intermediate counts are not available.}
  \label{tab:per-order-stats-even}
  \vspace{4pt}
  \begin{tabular}{c|r|r|r|r|r|r}
    \hline
    Order $n$ & Generated & After pruning & Class 1 / non-critical &
    Total critical & Overfull critical & Nontrivial survivors \\
    \hline
    $14$ & $43{,}418$ & $41{,}751$ & $41{,}750$ & $0$ & $0$ & $0$ \\
    $16$ & $666{,}854$ & $653{,}513$ & $653{,}512$ & $0$ & $0$ & $0$ \\
    $18$ & $11{,}780{,}245$ & $11{,}650{,}297$ & $11{,}650{,}296$ & $0$ & $0$ & $0$ \\
    $20$ & $232{,}275{,}898$ & $230{,}799{,}117$ & $230{,}799{,}116$ & $0$ & $0$ & $0$ \\
    $22$ & $5{,}022{,}269{,}988$ & $5{,}003{,}271{,}897$ & $5{,}003{,}271{,}895$ & $1$ & $0$ & $1$ \\
    $24$ & --- & --- & --- & --- & --- & $9$ \\
    \hline
  \end{tabular}
\end{table}

At the even orders $14,16,18$, and $20$, the pipeline finds no
$3$-critical graphs.  At order $22$, it leaves exactly one nontrivial survivor,
and at order $24$ exactly nine, agreeing with the known Brinkmann--Steffen
counts.  In every row with recorded intermediate counts, the identity
\[
  \text{Total critical}
  =
  \text{Overfull critical}
  +
  \text{Nontrivial survivors}
\]
holds, and the difference between ``After pruning'' and the sum of
``Class 1 / non-critical'' and ``Total critical'' is $1$ in each such row,
accounted for by the cycle $C_n$.

For order $22$, the final run generated $5{,}022{,}269{,}988$ $2$-connected
subcubic graphs and left exactly one nontrivial $3$-critical survivor.  The
source file is \texttt{results/order\_22\_delta\_3.json.gz}, with compressed
SHA-256 hash
\[
\texttt{57dbfccd9cb352564f5422530c9a0b7e269148c9789bf040f0dfd7ab96ed553e}.
\]
The edge-criticality and the absence of $3$-overfull subgraphs of this graph
were checked again after the census run; this recheck is not part of
the archived reports.

At order $24$ we used a native implementation of the same tests
(\texttt{code/native/}), preceded by a prefilter that discards graphs
violating two necessary conditions from Vizing's Adjacency Lemma: every
vertex has at least two neighbors of degree $3$, and $2n_2\le n_3$, where
$n_i$ is the number of vertices of degree $i$. The version of this implementation used at order $24$ streams the
output of \texttt{geng} and passes only the survivors on to the program
that writes the census file, so no counts of generated or pruned graphs
were recorded for this order; in the census file the fields
\texttt{generated\_biconnected} and \texttt{processed\_graphs}
therefore equal $9$, the number of survivors, and the remaining counters
likewise refer only to the survivors. The source archived in
\texttt{code/native/} is a later revision that also reports these
counts; it is the version used for the order-$19$ replay described
above.  The computation was split into
$1{,}024$ slices and used about $1{,}230$ allocated core-hours, as
recorded in the README.  It found
exactly nine nontrivial $3$-critical graphs, in agreement with Brinkmann
and Steffen~\cite{BrinkmannSteffen1997}.  Eight of them contain an
independent triangle and are vertex-blowups of the order-$22$ graph, and
the ninth is triangle-free.  The source file is
\texttt{results/order\_24\_delta\_3.json.gz}, with compressed SHA-256 hash
\[
\texttt{81dc57cef8680b4908e7159b186543209ce276b8fb3d6dff57e744ee68725523},
\]
and the edge-criticality and the absence of $3$-overfull subgraphs of
these nine graphs were checked again after the census run; this recheck
is not part of the archived reports.

\subsection{Categorization according to Theorem~\ref{thm:characterization}}
\label{S:census-map}

For each order $n\in\{13,15,17,19,21,22,23,24\}$, we assign each survivor $G$
to the first case of the proof of Theorem~\ref{thm:characterization} that
applies to it. If $G$ has an independent triangle, then it falls under
case~(a), the vertex-blowup case. Otherwise, if $G$ has an essential
$2$-edge-cut, which is always the case when $G$ has a triangle (see the
proof of Theorem~\ref{thm:characterization}), then it falls under
case~(b), the Haj\'os-join case, by Theorem~\ref{lem:Hajos2}.

Now suppose that $G$ has girth at least $4$ and is cyclically
$3$-edge-connected. If some cyclic $3$-edge-cut of $G$ has a shore
containing no vertex of degree $2$ in $G$, then, by
Theorem~\ref{lem:cyclic-cut-reduction} (the shore does not induce  a triangle,
since $G$ has girth at least $4$), contracting the graph induced on  this shore to a
single vertex gives a smaller nontrivial $3$-critical graph, and $G$ falls
under case~(c), the Meredith-type case. Otherwise, by
Lemma~\ref{lem:snark-reduction}, $G$ is a subgraph of a snark obtained as
described there. If $n$ is odd, then $G$ is obtained from a snark of order
$n+1$ by deleting one vertex and some independent edges, and it falls
under case~(d); if $n$ is even, then $G$ is obtained from a snark of order
$n$ by deleting some independent edges, and it falls under case~(e). When
testing for case~(c), we examine every cyclic $3$-edge-cut of $G$ and
both of its shores.

Applying this procedure to all survivors gives the partition recorded in
Table~\ref{tab:census-categorization}; the classification script is
\[
  \texttt{code/scripts/classify\_census\_characterization\_exact.py},
\]
and its output is
\[
  \texttt{reports/census\_characterization\_classification\_exact.json}.
\]

\begin{table}[!ht]
  \centering
  \scriptsize
  \setlength{\tabcolsep}{4pt}
  \caption{Census categorization of the nontrivial $3$-critical survivors by the
  first case of the proof of Theorem~\ref{thm:characterization} that applies, in the priority
  order vertex-blowup~(a), Haj\'os-join~(b), Meredith-type~(c),
  snark-completion~(d)/(e).  Each row is an exact partition: the four category
  counts sum to the survivor count.  Source data are in the data repository
  directory \texttt{results/}; the classification script and its output are
  described in the text.}
  \label{tab:census-categorization}
  \vspace{4pt}
  \begin{tabular}{c|r|r|r|r|r}
    \hline
    Order $n$ & Survivors & (a) Vertex-blowup & (b) Haj\'os-join &
    (c) Meredith-type & (d)/(e) Snark-completion \\
    \hline
    $13$ & $14$        & $10$        & $3$        & $1$        & $0$ \\
    $15$ & $94$        & $76$        & $13$       & $5$        & $0$ \\
    $17$ & $774$       & $639$       & $91$       & $36$       & $8$ \\
    $19$ & $6{,}984$   & $5{,}928$   & $712$      & $296$      & $48$ \\
    $21$ & $70{,}530$  & $60{,}479$  & $6{,}693$  & $3{,}061$  & $297$ \\
    $22$ & $1$         & $0$         & $0$        & $0$        & $1$ \\
    $23$ & $782{,}186$ & $672{,}397$ & $71{,}812$ & $36{,}069$ & $1{,}908$ \\
    $24$ & $9$         & $8$         & $0$        & $0$        & $1$ \\
    \hline
  \end{tabular}
\end{table}
For the odd orders,
the snark column is case~(d), while for the even orders $n=22$ and
$n=24$ it is case~(e). At every odd order, most survivors fall under the vertex-blowup
case~(a); for instance, this case accounts for $60{,}479$ of the $70{,}530$
survivors of order $21$ and for $672{,}397$ of the $782{,}186$ survivors of
order $23$. No survivor of order $13$ or $15$ falls under case~(d); this
is consistent with Theorem~\ref{thm:characterization}, since such a
survivor would come from a snark of order $14$ or $16$, and the Petersen
graph is the only snark of order less than $18$
(see~\cite{BrinkmannGoedgebeurHagglundMarkstrom2013}, which also covers
snarks of girth $4$, and see~\cite{Fouquet1982} for order $16$ and girth
at least $5$); we also checked with \texttt{geng} that the Petersen graph
is the only snark of order less than $18$.

The number of vertices of degree $2$ is also recorded for each survivor.
Every survivor of odd order that falls under case~(c) or case~(d) has
exactly three vertices of degree $2$, and the survivors with five
vertices of degree $2$, which first appear at order $17$, each fall
under case~(a) or~(b).
In particular, every survivor under case~(d) is obtained from a snark by
deleting one vertex and no edge. We verified this directly: for each of
the $8$, $48$, $297$, and $1{,}908$ survivors under case~(d) of
orders $17$, $19$, $21$, and $23$, respectively, adding a new vertex
adjacent to its three vertices of degree $2$ yields a cubic triangle-free
cyclically $4$-edge-connected graph of class~$2$, that is, a snark. Up to
isomorphism, the snarks obtained in this way have orders $18$, $20$, $22$,
and $24$ and number $2$, $6$, $31$, and $155$, respectively; recall that
our snarks may have girth $4$. These are exactly the numbers of snarks of
girth at least $4$ (called weak snarks there) of these orders reported
in~\cite{BrinkmannGoedgebeurHagglundMarkstrom2013}, and those of girth
at least $5$ number $2$, $6$, $20$, and $38$, which are exactly the
numbers of snarks of girth at least $5$ of these orders given there.
These isomorphism counts are not part of the archived reports.
The unique survivor of order $22$ and the triangle-free survivor of
order $24$, which are the only survivors under case~(e), have exactly
four vertices of degree $2$ each, and joining these four vertices by
two independent edges, for a suitable pairing, yields a snark; so each of them is
obtained from a snark by deleting exactly two independent edges. (The
other eight survivors of order $24$ also have four vertices of degree
$2$, as vertex-blowups of the order-$22$ graph.) These observations are
the basis for Conjecture~\ref{conj:snark-cases}; the evidence for its
even-order part is limited to these two graphs, which are the only
graphs of order at most $24$ that fall under case~(e).

Finally, we verified Conjecture~\ref{conj:petersen-minor} for all
nontrivial $3$-critical graphs of order at most $24$. Since $P^*$ has
maximum degree $3$, a graph contains $P^*$ as a minor if and only if it
contains a subdivision of $P^*$, and we searched for such subdivisions
by backtracking (script \texttt{code/scripts/verify\_pstar\_minor.py} in
the data repository, which tries every vertex of degree $3$ as the image
of a fixed branch vertex; its output is
\texttt{reports/pstar\_minor\_verification.json}). A direct search succeeded for every nontrivial
$3$-critical graph of order at most $17$ and for every survivor under
case~(d) or~(e) of orders $19$ to $24$. Every other survivor falls under
case~(a), (b), or~(c) and therefore contains a smaller nontrivial
$3$-critical graph as a minor, as explained before
Conjecture~\ref{conj:petersen-minor} in Section~\ref{sec:introduction};
hence all survivors contain $P^*$ as a minor by induction on the order.

\end{document}